\documentclass[11pt]{amsart}
\usepackage{amsmath,amscd,amssymb,amsthm,mathtools}
\usepackage[mathscr]{eucal}
\usepackage[margin=1in]{geometry}
\usepackage{booktabs}
\usepackage{longtable}
\usepackage{tabularx}
\usepackage{enumitem}
\usepackage{hyperref}

\numberwithin{equation}{section}

\newtheorem{theo}{Theorem}[section]
\newtheorem{coro}[theo]{Corollary}
\newtheorem{prop}[theo]{Proposition}
\newtheorem{lemm}[theo]{Lemma}
\newtheorem*{theoA}{Theorem A}
\newtheorem*{theoB}{Theorem B}
\newtheorem*{theoC}{Theorem C}
\newtheorem*{theoD}{Theorem D}
\newtheorem*{theoE}{Theorem E}
\theoremstyle{definition}
\newtheorem{defi}[theo]{Definition}
\newtheorem{exam}[theo]{Example}
\newtheorem{rema}[theo]{Remark}

\newcommand{\Mbar}{\overline{\mathcal M}}
\newcommand{\Q}{\mathbb Q}
\newcommand{\C}{\mathbb C}
\newcommand{\Z}{\mathbb Z}

\newcommand{\CC}{\mathbb{C}}
\newcommand{\PP}{\mathbb{P}}
\newcommand{\QQ}{\mathbb{Q}}
\newcommand{\ZZ}{\mathbb{Z}}
\newcommand{\V}{\mathbb V}
\newcommand{\bS}{\mathbf S}
\newcommand{\HH}{\mathbf{H}}
\newcommand{\Sbar}{\overline{\mathsf S}}
\newcommand{\Sym}{\operatorname{Sym}}
\newcommand{\tr}{\operatorname{tr}}
\newcommand{\str}{\operatorname{str}}
\newcommand{\ch}{\operatorname{ch}}
\newcommand{\Ind}{\operatorname{Ind}}
\newcommand{\rk}{\operatorname{rk}}
\newcommand{\bh}{\mathbf h}
\newcommand{\be}{\mathbf e}
\newcommand{\pl}{\circ}
\newcommand{\Lam}{\widehat\Lambda}
\newcommand{\sA}{\mathscr A}
\def\cA{{\mathcal A}}
\def\cC{{\mathcal C}}
\def\cE{{\mathcal E}}
\def\cF{{\mathcal F}}
\def\cG{{\mathcal G}}
\def\cH{{\mathcal H}}
\def\cL{{\mathcal L}}
\def\cM{{\mathcal M}}
\def\cT{{\mathcal T}}
\def\cU{{\mathcal U}}
\def\cX{{\mathcal X}}
\def\cY{{\mathcal Y}}
\def\bcM{\overline{\cM}}
\def\beq{\begin{equation}}
\def\eeq{\end{equation}}
\def\and{\quad{\rm and}\quad}
\def\lra{\longrightarrow }
\def\id{\mathrm{id}}
\def\Aut{\mathrm{Aut}}

\title[Hodge--Deligne polynomials and $\bS_n$-representations for 
$\overline{\mathcal M}_{2,n}$]{Hodge--Deligne polynomials and $\bS_n$-representations on the cohomology of $\overline{\mathcal M}_{2,n}$}

\author{Young-Hoon Kiem}
\address{School of Mathematics, Korea Institute for Advanced Study, 85 Hoegiro, Dongdaemun-gu, Seoul 02455, Korea}
\email{kiem@kias.re.kr}

\date{v2, November 2026}

\begin{document}
\begin{abstract}
We compute the Hodge-Deligne polynomials of the moduli spaces $\overline{\mathcal M}_{2,n}$ of stable $n$-pointed curves of genus two, together with the action of the symmetric group $\bS_n$ on their cohomology, for all $n$. Using Hassett's moduli spaces of weighted pointed curves, we first show that the generating function of the $\bS_n$-equivariant Hodge-Deligne polynomials of $\overline{\mathcal M}_{g,n}$ is obtained from that of the moduli spaces of minimally weighted pointed curves by plethysm with the explicit symmetric function $Z=h_1-\sum_{k\ge3}(x+x^2+\dots+x^{k-2})h_k$; numerically, this is an explicit change of variables. The latter spaces map to $\overline{\mathcal M}_g$, and over each boundary stratum they are built from punctured universal curves and Losev-Manin spaces. Combining the stratification by stable graphs with Burnside's formula and twisted characteristics, we obtain in genus two a closed formula whose only non-elementary input is the cohomology of symplectic local systems on $\mathcal M_{1,1}$ and $\mathscr A_2$, known by work of Eichler-Shimura and Petersen. This determines the $\bS_n$-representations on all $H^{p,q}(\overline{\mathcal M}_{2,n})$ for $n\le 19$ and on all even degree $H^{p,q}(\overline{\mathcal M}_{2,n})$ for every $n$. We carry out the computation in all degrees for $n\le 18$, where it agrees with an independent computation by point counts, and in even degrees for $n\le 30$. As an application, we find that the even degree cohomology of $\overline{\mathcal M}_{2,n}$ is of Hodge-Tate type if and only if $n\le 20$: the space $H^{23,1}(\overline{\mathcal M}_{2,21})$ is the multiplicity free $\bS_{21}$-representation with Frobenius characteristic $s_1\cdot h_2[e_{10}]$, of dimension $1939938$, so that $H^{24}(\overline{\mathcal M}_{2,21})$ is not spanned by algebraic cycle classes.
\end{abstract}
\maketitle

\section{Introduction}\label{S1}

The rational cohomology of the moduli spaces $\bcM_{g,n}$ of stable $n$-pointed curves of genus $g$ has been studied intensively, yet complete answers are known only in very low genus. In genus zero, the cohomology is generated by boundary divisors and is spanned by algebraic cycle classes \cite{Keel}, the Poincar\'e polynomials of all $\bcM_{0,n}$ are encoded in a single generating function \cite[Theorem 4.2.1]{Manin}, and the $\bS_n$-equivariant Hodge-Deligne polynomials were computed by Getzler \cite{Get95a}. In genus one, the even degree cohomology is tautological \cite{Pet0}, the $\bS_n$-equivariant Hodge-Deligne polynomials are determined by Getzler's work \cite{Get98}, and the Hodge-Deligne polynomials of all $\bcM_{1,n}$ were computed in \cite{Ki} by means of the Hassett moduli spaces of weighted pointed curves \cite{Hassett}.

In genus two, the picture becomes considerably richer. Odd degree cohomology appears, the tautological ring of $\bcM_{2,n}$ fails to be Gorenstein \cite{PT}, and the even degree cohomology is tautological precisely for $n<20$ \cite{Pet2}: a non-tautological algebraic class on $\bcM_{2,20}$ was constructed by Graber and Pandharipande \cite{GrP}. The $\bS_n$-equivariant Euler characteristics of $\bcM_{2,n}$ were computed by counting points over finite fields \cite{Ber2, Ber}.

The purpose of this paper is to give closed formulas for the generating functions of the Hodge-Deligne polynomials of $\bcM_{2,n}$ and of their $\bS_n$-equivariant refinements, for all $n$ simultaneously, and to draw some consequences for the Hodge structures and the $\bS_n$-representations on $H^*(\bcM_{2,n})$. Throughout, $\HH_X(u,v)$ denotes the Hodge-Deligne polynomial of $X$ and $x=uv$ (see \S\ref{S2}).

Let $\Lam$ be the completed ring of symmetric functions with coefficients in $\QQ[u,v]$, with power sums $p_k$, complete symmetric functions $h_k$, elementary symmetric functions $e_k$ and Schur functions $s_\lambda$. Plethysm on $\Lam$ is normalized by $p_k\pl u=u^k$ and $p_k\pl v=v^k$ (see \S\ref{S2.2}). For $t\in\{1,-1,x\}$ we write
\[
\bh(t)=\sum_{n\ge0}t^nh_n=\exp\Big(\sum_{k\ge1}\frac{t^kp_k}{k}\Big),\qquad
\be(t)=\sum_{n\ge0}t^ne_n=\exp\Big(\sum_{k\ge1}(-1)^{k-1}\frac{t^kp_k}{k}\Big),
\]
so that $\bh(1)\be(-1)=1$. For a variety or Deligne-Mumford stack $X$ with an action of $\bS_n$, we let $\HH^{\bS}_X\in\Lam$ be its $\bS_n$-equivariant Hodge-Deligne polynomial, i.e.\ the Frobenius characteristic of the alternating sum of the graded pieces of $H^*_c(X)$, weighted by $u^pv^q$ (Definition \ref{SHD}). The \emph{rank map} $p_1\mapsto y$, $p_k\mapsto0$ for $k\ge 2$ sends $\HH^{\bS}_X$ to $\HH_X\,y^n/n!$.

\subsection{From maximal to minimal weights}
Let $\bcM_{g,(0^+)^n}$ be the Hassett moduli stack of $n$-pointed curves of genus $g$ in which all marked points have weight $0^+$, i.e.\ a positive rational number sufficiently close to $0$. A point of $\bcM_{g,(0^+)^n}$ is a stable curve of genus $g$ with $n$ smooth points which are allowed to coincide, possibly with rational bridges carrying marked points, but without rational tails. The two families of moduli spaces determine each other through a single explicit substitution.

\begin{theoA}[Proposition \ref{7}, Theorem \ref{A}]
For $g\ge 2$, let
\[\cH^{\bS}_g=\sum_{n\ge0}\HH^{\bS}_{\bcM_{g,n}}\and \cL^{\bS}_g=\sum_{n\ge0}\HH^{\bS}_{\bcM_{g,(0^+)^n}}.\]
Then $\cL^{\bS}_g=\cH^{\bS}_g\pl Z$, equivalently $\cH^{\bS}_g=\cL^{\bS}_g\pl Z^{\langle-1\rangle}$, where
\begin{equation}\label{defZ}
Z=\frac{x^2(\bh(1)-1)-\bh(x)+1}{x(x-1)}=h_1-\sum_{k\ge3}\frac{x(x^{k-2}-1)}{x-1}\,h_k
\end{equation}
and $Z^{\langle-1\rangle}$ is its plethystic inverse. Numerically, let
\[\cH_g(z)=\sum_n\HH_{\bcM_{g,n}}\frac{z^n}{n!}\and \cL_g(y)=\sum_n\HH_{\bcM_{g,(0^+)^n}}\frac{y^n}{n!}.\]
Then $\cH_g(z)=\cL_g(y)$ where
\[ z=\frac{x^2e^y-x^2-e^{xy}+1}{x(x-1)}.\]
\end{theoA}

We give two proofs. The numerical statement follows from the factorization of $\bcM_{g,n}\to\bcM_{g,(0^+)^n}$ into explicit smooth blowups \cite{Ki} (Theorem \ref{9}). Since the blowup formula also holds for Chow groups, this also shows that the Hodge conjecture for $\bcM_{g,n}$ is equivalent to the Hodge conjecture for $\bcM_{g,(0^+)^n}$ (Proposition \ref{10}). The blowup centers are not preserved by $\bS_n$, and for the equivariant statement, we use instead the stratification by rational tails (Proposition \ref{tails}). This part is equivalent to the case $m=0$ of the change of variables of Kannan-Serpente-Yun \cite[Theorem A]{KSY}, which is stated in terms of the open moduli spaces $\cM_{0,k+1}$, and the new point is the closed form of $Z$, which follows from Getzler's formula for configuration spaces (Lemma \ref{Zclosed}). Remarkably, the same change of variables recovers Manin's generating function for the Poincar\'e polynomials of $\bcM_{0,n}$ (Remark \ref{Manin}).

\subsection{Minimally weighted pointed curves}
The advantage of $\bcM_{g,(0^+)^n}$ is that it comes with a forgetful morphism $\Phi_{g,n}:\bcM_{g,(0^+)^n}\to \bcM_g$ whose fibers are easy to describe. Over the stratum $\cM_\Gamma\subset\bcM_g$ of curves with dual graph $\Gamma$, the preimage is a quotient by $\Aut(\Gamma)$ of a disjoint union of products of fibered powers of punctured universal curves $\mathring\cC_v\to\cM_{g_v,d_v}$ and of Losev-Manin spaces $L_k$ \cite{LM} at the edges. For $\beta\in\bS_d$, let $A^{\bS,\beta}_{g,d}\in\Lam$ be the $\beta$-twisted characteristic of the fibered powers of $\mathring\cC_{g,d}=\cM_{g,d+1}\to\cM_{g,d}$, and let $C^{\bS}$ and $D^{\bS}$ be the untwisted and flip-twisted characteristics of the Losev-Manin spaces (Definition \ref{Adef}).

\begin{theoB}[Theorem \ref{graphsum}, Corollary \ref{24}, Proposition \ref{LM}]
For $g\ge2$, $\cL^{\bS}_g=\sum_{\Gamma}\cL^{\bS}_\Gamma$, the sum running over the stable graphs of genus $g$, where
\[
\cL^{\bS}_\Gamma=\frac1{|\Aut\Gamma|}\sum_{\alpha\in\Aut(\Gamma)}\ \prod_{O}p_{\ell(O)}\pl A^{\bS,\alpha^{\ell(O)}(v_O)}_{g_{v_O},d_{v_O}}\ \prod_{P}p_{\ell(P)}\pl B^{\alpha}_P,
\]
where $B^\alpha_P=C^{\bS}$ if $\alpha^{\ell(P)}$ preserves the ends of $e_P$ and $B^\alpha_P=D^{\bS}$ if $\alpha^{\ell(P)}$ flips $e_P$. Here the products run over the $\langle\alpha\rangle$-orbits $O$ of vertices and $P$ of edges, of lengths $\ell(O)$ and $\ell(P)$, with representatives $v_O\in O$ and $e_P\in P$, and $\alpha^{\ell}(v)$ is the permutation of the half-edges at $v$ induced by $\alpha^\ell$. Moreover
\[
C^{\bS}=\frac{x-1}{x-\bh(x)\be(-1)},\qquad D^{\bS}=\frac{\bh(x)\be(1)+x}{x+1}\cdot\big(p_2\pl C^{\bS}\big),
\]
whose images under the rank map are $C(y)=\dfrac{x-1}{x-e^{(x-1)y}}$ and $D(y)=\dfrac{e^{(x+1)y}+x}{x+1}$.
\end{theoB}
This theorem can also be derived from [KS, Theorem A] as the minimally weighted form of the graph sum (cf. Remark \ref{KSrem}).

Under the rank map, 
Theorem B reduces to a formula (Corollary \ref{24}) which expresses the generating function $\cL_g(y)$ in terms of the equivariant traces on the fibered powers of punctured universal curves. Thus the computation of $\HH^{\bS}_{\bcM_{g,n}}$, for every $g\ge 2$ and all $n$, is reduced to that of the twisted characteristics $A^{\bS,\beta}_{g',d}$ with $g'\le g$.

\subsection{Genus two}
For $g=2$, there are seven stable graphs, and all the remaining inputs can be computed. Let $V$ (resp.\ $\V$) be the local system $R^1q_*\QQ$ of the universal curve over $\cM_{1,1}$ (resp.\ $\cM_2$), and let $\mathbb S_\lambda$ denote the Schur functor. Let
\[
\cE^{\bS}=\sum_{|\lambda|\ \mathrm{even}}s_{\lambda'}\,\HH_{\cM_{1,1};\mathbb S_\lambda V},\qquad
\cF^{\bS}=\sum_{|\lambda|\ \mathrm{even}}s_{\lambda'}\,\HH_{\cM_{2};\mathbb S_\lambda \V},
\]
whose images under the rank map are $\cE(y)=\sum_{k}\HH_{\cM_{1,1};V^{\otimes k}}\frac{y^k}{k!}$ and $\cF(y)=\sum_k\HH_{\cM_{2};\V^{\otimes k}}\frac{y^k}{k!}$. Here $\lambda'$ denotes the conjugate (transpose) partition of $\lambda$. 
The series $\cE^{\bS}$ is determined by the Eichler-Shimura theory, i.e.\ by the dimensions of the spaces of elliptic cusp forms (Proposition \ref{40}), while $\cF^{\bS}$ is determined by Petersen's computation \cite{Pet} of the cohomology of symplectic local systems on $\sA_2$, i.e.\ by the dimensions of the spaces of elliptic cusp forms and of vector valued Siegel cusp forms for $\mathrm{Sp}_4(\ZZ)$, together with the branching from $\mathrm{GL}_4$ to $\mathrm{Sp}_4$ (Proposition \ref{EF}). We obtain an explicit formula for $\cL^{\bS}_2$ in terms of $\bh(t)$, $\be(t)$, $C^{\bS}$, $D^{\bS}$, $\cE^{\bS}$ and $\cF^{\bS}$ (Theorem \ref{L2S}). Its numerical form is as follows.

\begin{theoC}[Theorem \ref{L2S}, Corollary \ref{37}]
The generating function $\cL_2(y)$ is the sum of the following seven terms, which are the contributions of the seven strata of $\bcM_2$:
\begin{align*}
\cL_2={}&e^{(1+x)y}\mathcal F
+\tfrac12\Big[e^{2xy}\mathcal E^2C+x^2D\Big]
+\tfrac12\Big[e^{(x-1)y}\big(x\mathcal E+\mathcal E'\big)C
              +e^{(1+x)y}\big(x\mathcal E-\mathcal E'\big)D\Big]\\
&+\tfrac12\Big[e^{(2x-2)y}\mathcal EC^2+e^{2xy}\mathcal ECD\Big]\\
&+\tfrac18\Big[(x-2)e^{(x-3)y}C^2+2x\,e^{(x-1)y}CD
              +(x-2)e^{(x+1)y}D^2+(4x-4)e^{(x+1)y}\Big]\\
&+\tfrac1{12}\Big[e^{2(x-2)y}C^3+3e^{2xy}C+2e^{2(x+1)y}
              +D^3+3D+2\Big]\\
&+\tfrac18\Big[e^{2(x-2)y}C^3+2e^{(2x-2)y}C^2D
              +e^{2xy}CD^2+4D\Big].
\end{align*}
\end{theoC}

Combining Theorems A, B and \ref{L2S}, we obtain $\HH^{\bS}_{\bcM_{2,n}}$ for every $n$ for which the relevant dimensions $s_{j,k}=\dim S_{j,k}(\mathrm{Sp}_4(\ZZ))$ of spaces of vector valued Siegel cusp forms are known. For $n\le 19$, only six of them are nonzero, and they are all known, so we obtain the $\bS_n$-representations on all $H^{p,q}(\bcM_{2,n})$ for $n\le 19$ unconditionally. Moreover, the numbers $s_{j,k}$ affect only odd degrees, so the $\bS_n$-representations on all even degree $H^{p,q}(\bcM_{2,n})$ are determined for every $n$ (Corollary \ref{62}). In \S\ref{S7}, we list the Betti numbers of $\bcM_{2,n}$ for $n\le 19$, the even degree Betti numbers for $20\le n\le 30$, and the $\bS_n$-representations for small $n$. We carried out the $\bS_n$-equivariant computation in all degrees for $n\le 18$ and in even degrees for $n\le 30$. For $n\le18$, it agrees, Schur function by Schur function and Hodge type by Hodge type, with the $\bS_n$-equivariant Euler characteristics computed in \cite{Ber} by counting points over finite fields.

\subsection{Non-Tate classes in even degree}
Theorem \ref{L2S} computes Hodge numbers, not only Betti numbers, and this extra information is not vacuous. By \cite{Pet2}, the even degree cohomology of $\bcM_{2,n}$ is tautological, hence spanned by algebraic classes of Tate type, for $n<20$. We find that the Hodge-Tate property persists for one more step and then fails.

\begin{theoD}[Theorem \ref{59}, Corollary \ref{61}]
For $n\le 20$, the Hodge numbers $h^{p,q}(\bcM_{2,n})$ with $p+q$ even vanish unless $p=q$. On the other hand, the only off-diagonal even degree Hodge pieces of $\bcM_{2,21}$ are $H^{23,1}$ and $H^{1,23}$, and as $\bS_{21}$-representations
\[
\ch_{21}\,H^{23,1}(\bcM_{2,21})=\ch_{21}\,H^{1,23}(\bcM_{2,21})=s_1\cdot h_2[e_{10}]
=\sum_{k=0}^{10}s_{2^k1^{21-2k}}+\sum_{a\in\{1,3,5,7,9\}}s_{3\,2^a1^{18-2a}},
\]
so that $h^{23,1}(\bcM_{2,21})=h^{1,23}(\bcM_{2,21})=1939938$. Consequently $H^{24}(\bcM_{2,21})$ is not spanned by algebraic cycle classes, and
\[RH^{24}(\bcM_{2,21})\subsetneq H^{24}_{\mathrm{alg}}(\bcM_{2,21})\subsetneq H^{24}(\bcM_{2,21}).\]
\end{theoD}

Here the classes of type $(23,1)$ come from $\Sym^2\mathsf S_{12}(-1)$, where $\mathsf S_{12}$ is the motive of the discriminant cusp form $\Delta$ of weight $12$, and $H^{23,1}(\bcM_{2,21})\cong\Ind^{\bS_{21}}_{\bS_1\times(\bS_{10}\wr\bS_2)}\big(\mathbf 1\boxtimes(\mathrm{sgn}_{10}\boxtimes\mathrm{sgn}_{10})\big)$. The first strict inclusion is due to the non-tautological bielliptic cycles of van Zelm \cite{vZ}. Since the forgetful morphisms are surjective, the pullbacks $H^{24}(\bcM_{2,21})\to H^{24}(\bcM_{2,n})$ are injective for $n\ge 21$, so the even degree cohomology of $\bcM_{2,n}$ is of Hodge-Tate type if and only if $n\le 20$. Further non-Tate classes appear for larger $n$. For instance, $H^{1,27}(\bcM_{2,25})\cong H^{12,16}(\bcM_{2,25})$ has Frobenius characteristic $s_1e_{10}e_{14}$ and comes from $\mathsf S_{12}\otimes\mathsf S_{16}(-1)$ (Proposition \ref{n25}). Finally, since Theorem \ref{L2S} determines all $\bS_n$-representations $H^{p,p}(\bcM_{2,n})$, verifying the Hodge conjecture for $\bcM_{2,n}$ amounts to producing enough algebraic cycles, whose number and symmetry type are now explicitly known.

\subsection{Odd degree cohomology}
The computation also describes the odd degree cohomology as a representation of $\bS_n$.

\begin{theoE}[\S\ref{S7.4}]
The first nonzero odd degree cohomology group of $\bcM_{2,n}$ is $H^{13}(\bcM_{2,10})=H^{12,1}\oplus H^{1,12}$, and each of the two summands is the sign representation of $\bS_{10}$. For $n=11$, $H^{12,1}(\bcM_{2,11})\cong(V_{21^9}\oplus V_{1^{11}})^{\oplus2}$. The first classes of type $(p,0)$ appear in $H^{17}(\bcM_{2,14})$, where $H^{17,0}(\bcM_{2,14})\cong V_{2^7}$, of dimension $429$, comes from the Saito-Kurokawa lift of the cusp form of weight $18$.
\end{theoE}

\subsection*{Layout}
In \S\ref{S2}, we collect basic facts on Hodge-Deligne polynomials, equivariant traces, symmetric functions and twisted characteristics. In \S\ref{S3}, we recall the blowup description of \cite{Ki} and prove Theorem A and the comparison of the Hodge conjectures. In \S\ref{S4}, we stratify $\bcM_{g,(0^+)^n}$ by stable graphs, prove Theorem B and compute the contributions of the Losev-Manin spaces. In \S\ref{S5}, we compute the genus two inputs: the cohomology of local systems and the twisted characteristics of punctured universal curves. In \S\ref{S6}, we assemble these into the genus two formula and prove Theorem C. In \S\ref{S7}, we present the numerical results and prove Theorems D and E.

\subsection*{Conventions}
All the schemes or stacks in this paper are defined over $\CC$. All the cohomology groups are defined over $\QQ$ unless stated otherwise. Graded objects carry the Koszul sign rule, so that $\HH^{\bS}$ is a homomorphism of $\lambda$-rings from the Grothendieck ring of $\bS$-spaces (see \S\ref{S2.2}).

\subsection*{Acknowledgement} It is my pleasure to thank Jonas Bergstr\"om, Jinwon Choi and Han-Bom Moon for useful discussions and comments. I thank Terry Song and Siddarth Kannan for kindly informing me of \cite{KS} shortly after the first version was posted on the arXiv.

\bigskip
\section{Preliminaries}\label{S2}

\subsection{Hodge-Deligne polynomials and Burnside's formula}\label{S2.1}
For a Deligne-Mumford stack $X$ with quasi-projective coarse moduli space, let $\HH_X$ denote the \emph{Hodge-Deligne polynomial}
$$\HH_X(u,v)=\sum_k(-1)^k\sum_{p,q}u^pv^q\dim H^{p,q;k}_c(X), \quad H_c^{p,q;k}(X)=\mathrm{gr}^p_F\mathrm{gr}^W_{p+q}H^k_c(X)$$
where $\mathrm{gr}_F$ and $\mathrm{gr}^W$ denote the graded parts with respect to the Hodge and weight filtrations respectively. The Hodge-Deligne polynomial is \emph{motivic} in the sense that
\begin{enumerate}
    \item $\HH_X=\HH_{X-Z}+\HH_Z$ for closed $Z\subset X$;
    \item $\HH_{X\times Y}=\HH_X\cdot \HH_Y.$
\end{enumerate}
It moreover satisfies the blowup formula
$$\HH_X=\HH_Y+\HH_Z\cdot (\HH_{\PP^{c-1}}-1)=\HH_Y+\HH_Z\cdot\frac{x^c-x}{x-1}$$
when $X\to Y$ is the blowup of a smooth $Y$ along a smooth center $Z$ of codimension $c\ge 2$.
The \emph{Poincar\'e polynomial} of $X$ is
$$P_X(t)=\HH_X(-t,-t)\in \ZZ[t].$$
When $X$ is smooth and proper of dimension $d$, the Poincar\'e polynomial and the even degree Poincar\'e polynomial may be defined as
\beq\label{12}
P_X(t)=\sum_i t^i\dim H^i(X),\quad P^{\mathrm{even}}_X(x)=\sum_k x^k\dim H^{2k}(X).\eeq
In this case, the Poincar\'e duality reads
$$\HH_X(u,v)=(uv)^d\HH_X(u^{-1},v^{-1}),\quad P_X(t)=t^{2d}P_X(t^{-1}), \quad P_X^{\mathrm{even}}(x)=x^dP_X^{\mathrm{even}}(x^{-1}).$$
We use $x=uv$ throughout and a class of Tate type $(m,m)$ contributes $x^m$.

\begin{defi}\label{13}
For a finite group $G$ acting on $X$ and $g\in G$, we write
\[
\HH^g_X=\sum_k(-1)^k\sum_{p,q}u^pv^q\tr\big(g^*|_{H^{p,q;k}_c(X)}\big)
\]
for the \emph{equivariant trace}, a polynomial in $u,v$
by functoriality of the mixed Hodge structure.
For a local system $V$ on $X$ which is a polarizable variation of Hodge structure, we write
$$\HH_{X;V}(u,v)=\sum_k(-1)^{k}\sum_{p,q}u^pv^q\dim H^{p,q;k}_c(X;V),$$
and similarly $\HH^g_{X;V}$ for a $G$-equivariant $V$. By additivity, $\HH_{X;V}$ and $\HH^g_{X;V}$ make sense for any class $V$ in the Grothendieck group of (equivariant) variations of mixed Hodge structure, where a shift $[-1]$ contributes a sign.
\end{defi}

Since we use $\Q$-coefficients, we have $H^{p,q;k}_c(X/G)=H^{p,q;k}_c(X)^G$ and by taking the trace of the projection $$\frac1{|G|}\sum_{g\in G}g^*:H^{p,q;k}_c(X)\lra H^{p,q;k}_c(X)^G\subset H^{p,q;k}_c(X)$$ to the invariant part (cf. \cite[Proposition 2.8]{FuHa}), we obtain Burnside's formula
\beq\label{14}
\HH_{X/G}=\frac1{|G|}\sum_{g\in G}\HH^g_X
\eeq
which we will use repeatedly below.
We will also repeatedly use the following two elementary facts.
\begin{lemm}\label{15}
Let $X=\bigsqcup_\alpha X_\alpha$ be a finite stratification
which is preserved by $G$ up to permutation,
i.e.\ $g\in G$ permutes the $X_\alpha$.  Then
$\HH^g_X=\sum_{\alpha:\,gX_\alpha=X_\alpha}\HH^g_{X_\alpha}$.
\end{lemm}
\begin{proof}
Let $H=\langle g\rangle$. The equivariant trace is additive for $H$-stable closed subsets, since the long exact sequence of compactly supported cohomology for $Z\subset X$ is $H$-equivariant and compatible with the mixed Hodge structures. Ordering the strata compatibly with their closures, we may therefore assume that $X$ is the union of a single $H$-orbit $X_{\alpha}, gX_\alpha,\dots,g^{m-1}X_\alpha$ of strata, which are open and closed in $X$ since none of them meets the closure of another. Then
$H^*_c(X)=\bigoplus_{i=0}^{m-1}H^*_c(g^iX_\alpha)$ and $g$ permutes the summands cyclically.
If $m>1$, the matrix of $g^*$ has zero diagonal blocks, so its trace vanishes. Only strata with
$gX_\alpha=X_\alpha$ survive.
\end{proof}

\begin{lemm}\label{16}
Let $\iota$ be the inversion $z\mapsto z^{-1}$ on $\C^*$.  Then $\HH^\iota_{\C^*}=x+1$,
and for the diagonal inversion on $(\C^*)^m$ one has $\HH^\iota_{(\C^*)^m}=(x+1)^m$.
\end{lemm}
\begin{proof}
From $0\to H^0(\PP^1)\to H^0(\{0,\infty\})\to H^1_c(\C^*)\to H^1(\PP^1)=0$, we get
$H^1_c(\C^*)=\Q$ spanned by the class of $[0]-[\infty]$, of weight $0$, on which $\iota$
acts by $-1$ since it exchanges $0$ and $\infty$. We also have  $H^2_c(\C^*)=H^2(\PP^1)=\Q(-1)$, on
which $\iota$ acts trivially, being a degree-one self-map of $\PP^1$.  Hence
$\HH^\iota_{\C^*}=-(-1)+x=x+1$.  The second claim follows by the K\"unneth formula.
\end{proof}

\subsection{Symmetric functions and equivariant Hodge-Deligne polynomials}\label{S2.2}
Let $\Lambda$ be the ring of symmetric functions over $\QQ$, with power sums $p_k$, and let $\Lam=\QQ[u,v][[p_1,p_2,\dots]]$ be its completion with coefficients in $\QQ[u,v]$, graded by $\deg p_k=k$. For a partition $\lambda$ we write $p_\lambda=\prod_ip_{\lambda_i}$, and for $\sigma\in\bS_n$ we write $\lambda(\sigma)$ for its cycle type. The \emph{Frobenius characteristic} of a finite dimensional representation $W$ of $\bS_n$ is
\[
\ch_n(W)=\frac1{n!}\sum_{\sigma\in\bS_n}\tr(\sigma|W)\,p_{\lambda(\sigma)}=\sum_{\lambda\vdash n}[W:V_\lambda]\,s_\lambda,
\]
where $V_\lambda$ is the Specht module and $s_\lambda$ the Schur function. The trivial and the sign representations have characteristics $h_n$ and $e_n$, and $\ch$ takes induction products to products.

\begin{defi}\label{SHD}
For a Deligne-Mumford stack $X$ with an action of $\bS_n$, the \emph{$\bS_n$-equivariant Hodge-Deligne polynomial} is
\[
\HH^{\bS}_X=\frac1{n!}\sum_{\sigma\in\bS_n}\HH^\sigma_X\,p_{\lambda(\sigma)}=\sum_{k,p,q}(-1)^k\ch_n\big(H^{p,q;k}_c(X)\big)\,u^pv^q\in\Lam.
\]
If $X$ is smooth and proper and $H^k(X)$ is pure of weight $k$ for all $k$, then $(-1)^{p+q}$ times the coefficient of $u^pv^qs_\lambda$ in $\HH^{\bS}_X$ is the multiplicity of $V_\lambda$ in $H^{p,q}(X)=H^{p,q;p+q}(X)$.
\end{defi}

An \emph{$\bS$-space} $\cX$ is a sequence of Deligne-Mumford stacks $\cX(n)$ with actions of $\bS_n$, $n\ge0$, and we write $\HH^{\bS}_\cX=\sum_n\HH^{\bS}_{\cX(n)}$. The species (or Day convolution) product is
\[(\cX\cdot\cY)(n)=\coprod_{I\sqcup J=[n]}\cX(I)\times\cY(J),\]
and if $\cY(0)=\emptyset$, the composition is $(\cX\pl\cY)(n)=\coprod_j\big(\cX(j)\times\cY^{\cdot j}(n)\big)/\bS_j$. Plethysm on $\Lam$ is the operation $(f,g)\mapsto f\pl g$, defined when $g$ has no constant term or $f$ is a polynomial, which is additive and multiplicative in $f$ and satisfies
\[p_k\pl p_m=p_{km},\quad p_k\pl u=u^k,\quad p_k\pl v=v^k,\]
\[p_k\pl(g_1+g_2)=p_k\pl g_1+p_k\pl g_2,\quad p_k\pl(g_1g_2)=(p_k\pl g_1)(p_k\pl g_2).\]
Thus if $f=\sum_\lambda c_\lambda(u,v)p_\lambda$ then $f\pl g=\sum_\lambda c_\lambda(u,v)\prod_i (p_{\lambda_i}\pl g)$: the coefficients of $f$ are left alone and those of $g$ are acted on by the Adams operations $\psi^k(u^pv^q)=u^{kp}v^{kq}$. With these conventions, and the Koszul sign rule for the action of $\bS_j$ on tensor products of graded vector spaces,
\beq\label{lambdaring}
\HH^{\bS}_{\cX\cdot\cY}=\HH^{\bS}_{\cX}\HH^{\bS}_{\cY},\qquad \HH^{\bS}_{\cX\pl\cY}=\HH^{\bS}_{\cX}\pl\HH^{\bS}_{\cY}
\eeq
by \cite{Get95b} and \cite[\S5]{GeP}. This uses only the K\"unneth formula and \eqref{14}.

The \emph{rank map} $\rk:\Lam\to\QQ[u,v][[y]]$ is the ring homomorphism with $p_1\mapsto y$ and $p_k\mapsto0$ for $k\ge2$. It sends $\ch_n(W)$ to $\dim(W)\,y^n/n!$, hence $\HH^{\bS}_X$ to $\HH_X\,y^n/n!$ for an $\bS_n$-stack $X$.

\begin{lemm}\label{rankmap}
Let $f,g\in\Lam$ and let $g(0)\in\QQ[u,v]$ be the constant term of $g$. Then $\rk(p_1\pl g)=\rk(g)$ and $\rk(p_k\pl g)=\psi^k(g(0))$ for $k\ge2$. In particular, if $g(0)=0$ then $\rk(f\pl g)=\rk(f)\big(\rk(g)\big)$, the composition of power series in $y$.
\end{lemm}
\begin{proof}
For $k\ge2$, $p_k\pl g$ is obtained from $g$ by applying $\psi^k$ to the coefficients and replacing each $p_m$ by $p_{km}$, so all its terms of positive degree are killed by $\rk$. If $g(0)=0$, only the terms $c_{1^m}p_1^m$ of $f$ survive, and they give $\sum_mc_{1^m}\rk(g)^m$.
\end{proof}

\subsection{Twisted characteristics}\label{S2.3}
Let $\cX$ be an $\bS$-space and $\alpha$ a finite order automorphism of $\cX$, i.e.\ automorphisms of the $\cX(n)$ commuting with the $\bS_n$-actions. We define the \emph{$\alpha$-twisted characteristic}
\[
\HH^{\bS,\alpha}_{\cX}=\sum_{n\ge0}\frac1{n!}\sum_{\sigma\in\bS_n}\HH^{\sigma\alpha}_{\cX(n)}\,p_{\lambda(\sigma)}\in\Lam.
\]
For $\alpha=\id$, it is $\HH^{\bS}_{\cX}$, and $\rk(\HH^{\bS,\alpha}_{\cX})=\sum_n\HH^\alpha_{\cX(n)}y^n/n!$. If a finite group $G$ acts on $\cX$ commuting with the $\bS_n$-actions, then
\beq\label{Burnside}
\HH^{\bS}_{\cX/G}=\frac1{|G|}\sum_{\alpha\in G}\HH^{\bS,\alpha}_{\cX},
\eeq
by \eqref{14} applied to the isotypic components, or equivalently by applying \eqref{14} to the group $\bS_n\times G$. The same definitions apply to $\bS$-objects in any category with traces, such as graded vector spaces with mixed Hodge structures.

\begin{lemm}\label{strata}
Let $X=\bigsqcup_\beta X_\beta$ be a finite stratification of an $\bS_n$-space $X$ by locally closed pieces, permuted by $\bS_n\times\langle\alpha\rangle$. Then
\[\frac1{n!}\sum_\sigma\HH^{\sigma\alpha}_X p_{\lambda(\sigma)}=\frac1{n!}\sum_\sigma\sum_{\beta:\,\sigma\alpha X_\beta=X_\beta}\HH^{\sigma\alpha}_{X_\beta}p_{\lambda(\sigma)}.\]
\end{lemm}
\begin{proof}
Apply Lemma \ref{15} to each $\sigma\alpha$.
\end{proof}

Unlike in the non-equivariant setting, a stratum moved by $\alpha$ may be fixed by $\sigma\alpha$ for suitable $\sigma\in\bS_n$. The following lemma, a form of the classical relation between wreath products and plethysm \cite[Appendix A]{Mac}, \cite[\S7]{GK}, accounts for this.

\begin{lemm}\label{orbit}
Let $\cY$ be an $\bS$-space with a finite order automorphism $\gamma$, and let $\cX=\cY^{\cdot\ell}$ be the species product of $\ell$ copies of $\cY$, with the automorphism $\alpha$ which maps the $i$-th factor to the $(i+1)$-st factor by the identity for $i<\ell$ and the $\ell$-th factor to the first by $\gamma$. Then
\[\HH^{\bS,\alpha}_{\cX}=p_\ell\pl\HH^{\bS,\gamma}_{\cY}.\]
\end{lemm}
\begin{proof}
A point of $\cX(n)$ is an ordered decomposition $[n]=I_1\sqcup\dots\sqcup I_\ell$ together with points of $\cY(I_i)$. By Lemma \ref{strata}, for $\sigma\in\bS_n$, only the strata with $\sigma(I_i)=I_{i+1}$ (indices mod $\ell$) contribute to the trace of $\sigma\alpha$. Then $|I_i|=m$ for all $i$ and $n=\ell m$. Identifying all $I_i$ with $I_1$ via $\sigma$, the map $\sigma\alpha$ acts on $H^*_c(\cY(I_1))^{\otimes\ell}$ as the cyclic shift composed with $(\tau\gamma,1,\dots,1)$, where $\tau=\sigma^\ell|_{I_1}$. For a graded vector space $W$ with Hodge structures and a graded endomorphism $A$ preserving Hodge types, the super trace of the cyclic shift composed with $(A,1,\dots,1)$ on $W^{\otimes\ell}$ is $\psi^\ell(\str A)$. In a homogeneous basis $(w_i)$, only the vectors $w_i^{\otimes\ell}$ contribute to the trace, with coefficient $A_{ii}$, Koszul sign $(-1)^{|w_i|(\ell-1)}$ and parity $(-1)^{\ell|w_i|}$, so the total sign is $(-1)^{|w_i|}$, and the Hodge type of $w_i^{\otimes\ell}$ is $\ell$ times that of $w_i$. Moreover the cycle type of $\sigma$ is $\ell$ times that of $\tau$. There are $n!/(m!)^{\ell}$ ordered decompositions with $|I_i|=m$, and for each of them and each $\tau\in\bS(I_1)$, there are $(m!)^{\ell-1}$ permutations $\sigma$ with $\sigma(I_i)=I_{i+1}$ and $\sigma^\ell|_{I_1}=\tau$. Hence
\[\frac{1}{n!}\sum_\sigma\HH^{\sigma\alpha}_{\cX(n)}p_{\lambda(\sigma)}=\frac1{m!}\sum_{\tau\in\bS_m}\psi^\ell\big(\HH^{\tau\gamma}_{\cY(m)}\big)p_{\ell\lambda(\tau)},\]
which is the degree $n=\ell m$ part of $p_\ell\pl\HH^{\bS,\gamma}_{\cY}$.
\end{proof}

\begin{lemm}\label{fibered}
Let $q:Y\to B$ be a representable morphism of Deligne-Mumford stacks and $\alpha$ a finite order automorphism of $q$, i.e.\ compatible automorphisms of $Y$ and $B$. Let $Y^{(r)}_B$ be the $r$-fold fibered power, with $\bS_r$ permuting the factors and $\alpha$ acting diagonally, and let $Y^{(\bullet)}_B$ be the $\bS$-space $r\mapsto Y^{(r)}_B$. Let $K=Rq_!\QQ$, an $\alpha$-equivariant object whose class lies in the Grothendieck group of $\alpha$-equivariant variations of mixed Hodge structure on $B$. Then
\[
\HH^{\bS,\alpha}_{Y^{(\bullet)}_B}=\HH^{\alpha}_{B;\sum_r h_r[K]},\qquad \sum_r h_r[K]=\sum_\lambda s_\lambda\otimes\mathbb S_\lambda(K),
\]
where $\mathbb S_\lambda$ is the super Schur functor and $\HH^\alpha_{B;-}$ is applied to the coefficients of the $s_\lambda$. In particular:
\begin{enumerate}[label=(\alph*)]
\item if $K=T\oplus K'$ with $T$ a constant variation on which $\alpha$ acts by a fixed automorphism, then
\[\HH^{\bS,\alpha}_{Y^{(\bullet)}_B}=\exp\Big(\sum_{k\ge1}\psi^k\big(\str(\alpha^k|T)\big)\frac{p_k}k\Big)\cdot\HH^\alpha_{B;\sum_rh_r[K']};\]
\item if $K'=W[-1]$ with $W$ a variation of pure Hodge structure placed in degree $0$, then $\sum_r h_r[W[-1]]=\sum_\lambda(-1)^{|\lambda|}s_{\lambda'}\otimes\mathbb S_\lambda W$.
\end{enumerate}
\end{lemm}
\begin{proof}
By the relative K\"unneth formula, $R(q^{(r)})_!\QQ=K^{\otimes r}$, with $\bS_r$ permuting the factors with the Koszul sign and $\alpha$ acting diagonally. Taking compactly supported cohomology of $B$ with coefficients in the $\bS_r\times\langle\alpha\rangle$-equivariant object $K^{\otimes r}$, through the Leray spectral sequence which does not change traces, and Frobenius characteristics gives the first statement, as in \cite[\S5]{Get95b} and \cite[\S5]{GeP}. For (a), $\sum_rh_r[T\oplus K']=\sum_rh_r[T]\cdot\sum_rh_r[K']$, the cohomology of $B$ with coefficients in $T^{\otimes a}\otimes K'^{\otimes b}$ is $T^{\otimes a}\otimes H^*_c(B;K'^{\otimes b})$ since $T$ is constant, and $\sum_rh_r[T]=\exp(\sum_k p_k\otimes\psi^k(T)/k)$ in the $\lambda$-ring of $\alpha$-equivariant graded objects, where $\str(\alpha|\psi^kT)=\psi^k\str(\alpha^k|T)$. For (b), use the Cauchy identity $\sum_r h_r[XY]=\sum_\lambda s_\lambda[X]s_\lambda[Y]$ and $s_\lambda[-Y]=(-1)^{|\lambda|}s_{\lambda'}[Y]$.
\end{proof}

\begin{lemm}\label{Cstar}
Let $Q_m=(\CC^*)^m/\CC^*$ with $\bS_m$ permuting the coordinates and $\iota$ the diagonal inversion, and let $Q$ be the $\bS$-space $m\mapsto Q_m$ for $m\ge1$, with $Q(0)=\emptyset$. Then
\[
\HH^{\bS}_{Q}=\frac{\bh(x)\be(-1)-1}{x-1},\qquad \HH^{\bS,\iota}_{Q}=\frac{\bh(x)\be(1)-1}{x+1}.
\]
\end{lemm}
\begin{proof}
$H^*_c(\CC^*)$ consists of $\QQ(-1)$ in degree $2$, with trivial $\iota$-action, and a class of weight $0$ in degree $1$, on which $\iota$ acts by $-1$ (Lemma \ref{16}). So $\str(\iota^k|H^*_c(\CC^*))=x-(-1)^k$. By Lemma \ref{fibered} (a) with $B$ a point, $\sum_m\HH^{\bS}_{(\CC^*)^m}=\exp(\sum_k(x^k-1)p_k/k)=\bh(x)\be(-1)$ and $\sum_m\HH^{\bS,\iota}_{(\CC^*)^m}=\exp(\sum_k(x^k-(-1)^k)p_k/k)=\bh(x)\be(1)$. For $m\ge1$, the quotient map $(\CC^*)^m\to Q_m$ is a trivial principal bundle under the diagonal $\CC^*$, which commutes with $\bS_m$ and is normalized by $\iota$, inducing the inversion on $\CC^*$. Hence $\HH^{\sigma}_{(\CC^*)^m}=(x-1)\HH^\sigma_{Q_m}$ and $\HH^{\sigma\iota}_{(\CC^*)^m}=(x+1)\HH^{\sigma\iota}_{Q_m}$ for all $\sigma\in\bS_m$.
\end{proof}

\bigskip
\section{From maximal to minimal weights}\label{S3}

In this section, we recall the Hassett moduli stacks of weighted pointed curves from \cite{Hassett} and the factorization of the natural morphism
\beq\label{11} \rho:\bcM_{g,n}\lra \bcM_{g,(0^+)^n}\eeq
from the moduli space of maximally weighted pointed curves to that of minimally weighted pointed curves (Theorem \ref{9}) from \cite{Ki}. As a consequence, we derive a functional formula that compares the generating functions of their Hodge-Deligne polynomials (Proposition \ref{7}), and we show that the Hodge conjecture for one of these moduli spaces implies the same for the other. Finally, we prove the $\bS_n$-equivariant form of the comparison (Theorem \ref{A}) by a different argument.

\subsection{Hassett spaces and the blowup formula}
\begin{defi} \cite{Hassett}
    Let $g,n\ge 0$ and $\cA=(a_1,\cdots,a_n)\in (0,1]^n$ with $2g-2+\sum a_i>0$. A connected projective curve $C$ with marked points $p_1,\cdots,p_n\in C$ is called $\cA$-\emph{stable} if
    \begin{enumerate}
        \item $C$ has at worst nodal singularities and $p_i$ are all smooth points;
 \item if $p_{i_1}=\cdots =p_{i_k}$, then $a_{i_1}+\cdots+a_{i_k}\le 1$;
 \item $\omega_C(a_1p_1+\cdots+a_np_n)$ is ample.\end{enumerate}
\end{defi}

We let $\bcM_{g,\cA}$ be the moduli stack of $\cA$-stable $n$-pointed curves of genus $g$, which is a smooth proper Deligne-Mumford stack of dimension $3g-3+n$ with projective connected coarse moduli space by \cite{Hassett}.
When $a_i>\frac12$ for all $i$, $\bcM_{g,\cA}$ is the Knudsen-Deligne-Mumford moduli stack $\bcM_{g,n}$ of stable curves of genus $g$ with $n$ marked points.

Let $\cA_{r|s}=(1,\cdots,1,0^+,\cdots,0^+)$ where $1$ is repeated $r$ times and $0^+$ is repeated $s$ times. Here of course $0^+$ means a positive rational number sufficiently close to $0$.
Let
\beq\label{8} \bcM_{g,r|s}=\bcM_{g,\cA_{r|s}}, \quad \HH_{r|s}=\HH_{\bcM_{g,r|s}}.\eeq
In particular, $\HH_{n|0}=\HH_{\bcM_{g,n}}$ (resp. $\HH_{0|n}=\HH_{\bcM_{g,(0^+)^n}}$) is the Hodge-Deligne polynomial of the moduli stack $\bcM_{g,n}=\bcM_{g,n|0}$ (resp. $\bcM_{g,(0^+)^n}=\bcM_{g,0|n}$) of stable $n$-pointed curves of genus $g$ with maximal weights $1$ (resp. minimal weights $0^+$). By condition (3), an irreducible component of genus $0$ of a curve in $\bcM_{g,(0^+)^n}$ has at least two nodes, and it carries a marked point if it has exactly two. In other words, curves in $\bcM_{g,(0^+)^n}$ have no rational tails, and the marked points may coincide.

\begin{theo}\cite[Theorem 3.1]{Ki}\label{9} Fix $g\ge 1$ and let $x=uv$. For $r\ge 1$,
$\bcM_{g,r|s}$ is obtained from $\bcM_{g,r-1|s+1}$ by a sequence of smooth blowups whose centers are disjoint unions of $\binom{s}{k}$ copies of $\bcM_{g,r|k}$, of codimension $s-k$, for $0\le k\le s-2$. By the blowup formula \cite[p606]{GrHa}, we have
\beq\label{1}
\HH_{r|s}=\HH_{r-1|s+1}+\sum_{k=0}^{s-2}\binom{s}{k}\HH_{r|k}\frac{x^{s-k}-x}{x-1}.\eeq
\end{theo}
See \cite[\S3]{Ki} for a proof of Theorem \ref{9}. As a consequence, we obtain a clean formula \eqref{6} that compares $\HH_{n|0}$ with $\HH_{0|n}$.

\begin{prop}\label{7}
    For $g\ge 2$, let \beq\label{38}\cH_g(z)=\cH(z)=\sum_{n\ge 0}\HH_{n|0}\frac{z^n}{n!}, \quad \cL_g(y)=\cL(y)=\sum_{n\ge 0}\HH_{0|n}\frac{y^n}{n!}.\eeq
    Then we have
    \beq\label{6}\cH(z)=\cL(y)\eeq where
    \beq\label{5} z=\frac{x^2e^y-x^2-e^{xy}+1}{x(x-1)}=y-\sum_{k\ge 3}\frac{y^k}{k!}\frac{x(x^{k-2}-1)}{x-1}.\eeq
\end{prop}
We think of $x=uv$ as a constant. Note that $z=z(y)$ is an entire function of $y$ for every $x$ whose inverse function $y=y(z)$ is holomorphic in a neighborhood of $0$.
\begin{proof}
Let \beq\label{2}
\HH_r=\sum_{s\ge 0} \HH_{r|s} \frac{y^s}{s!}, \quad P=\sum_{s\ge 2}\frac{x^s-x}{x-1} \frac{y^s}{s!}=\frac{e^{xy}-xe^y+x-1}{x-1}.\eeq
Then \eqref{1} is equivalent to
\beq\label{3}
\HH_r=\frac{d}{dy}\HH_{r-1}+\HH_r\cdot P \quad \text{or} \quad
\HH_r=\frac{d\HH_{r-1}}{dy} \frac{1}{1-P}.\eeq
We introduce a new variable $\tau$ such that $dy/d\tau=1/(1-P)$:
$$\tau(y)=\int_0^y(1-P(w))dw=\int_0^y\frac{xe^w-e^{xw}}{x-1} dw = \frac{x^2e^y-x^2-e^{xy}+1}{x(x-1)}.$$
Since $\tau(y)=y+O(y^3)$, $\tau$ is a formal coordinate and $y=y(\tau)$ is its inverse power series. Then we have
\beq\label{4}
\HH_r=\frac{d\HH_{r-1}}{dy} \frac{1}{1-P}=\frac{d\HH_{r-1}}{dy} \frac{dy}{d\tau}=\frac{d\HH_{r-1}}{d\tau}=\frac{d^r\HH_{0}}{d\tau^r}.\eeq
Then by Taylor expansion, we have
$$\sum_{r\ge 0}\HH_r\frac{z^r}{r!}=\sum_{r\ge 0}\frac{d^r\HH_{0}}{d\tau^r} \frac{z^r}{r!}=\HH_0(y(\tau+z)).$$
Set $y=0$. Then we have $\tau=0$, $\HH_r=\HH_{r|0}$ and, writing $y$ for $y(z)$,
\[
\cH(z)=\sum_{n\ge 0}\HH_{n|0}\frac{z^n}{n!}=\HH_0(y)=\sum_{n\ge 0} \HH_{0|n}\frac{y^n}{n!}=\cL(y). \qedhere\]
\end{proof}

\begin{rema}\label{Manin}
If we let $\varphi=e^y-1$, we have
$$z=\frac{x^2\varphi-(1+\varphi)^{x}+1}{x(x-1)}.$$
By \cite[Theorem 4.2.1]{Manin}, the inverse is
$$\varphi=z+\sum_{n\ge 2}P_{\bcM_{0,n+1}}(\sqrt{x})\frac{z^n}{n!}$$
where $P_{\bcM_{0,n+1}}(q)$ denotes the Poincar\'e polynomial of $\bcM_{0,n+1}$. 
See Remark \ref{Arem} below. 
\end{rema}

\subsection{The Hodge conjecture}
Since the blowup formula holds for the Chow group as well as the ordinary cohomology group, we have the following as another application of Theorem \ref{9}.
\begin{prop}\label{10} Fix $g, n\ge 1$.
    The Hodge conjecture holds for $\bcM_{g,m}$ for all $m\le n$ if and only if it holds for $\bcM_{g,(0^+)^m}$ for all $m\le n$. If one of these two holds, then the Hodge conjecture holds for $\bcM_{g,r|s}$ for all $r+s\le n$.
\end{prop}
\begin{proof}
    Recall that the Hodge conjecture is about the surjectivity of the cycle class map $$A^k(X)_\QQ\lra H^{2k}(X)\cap H^{k,k}(X;\CC).$$
    Suppose that it holds for $\bcM_{g,n}$. For any $\xi\in H^{2k}(\bcM_{g,(0^+)^n})_\QQ\cap H^{k,k}(\bcM_{g,(0^+)^n})$, its pullback by the natural morphism \eqref{11} must be a Hodge class which is the image of some $\eta\in A^k(\bcM_{g,n})_\QQ$. Since $\rho$ is birational, we have $\rho_*\rho^*\xi=\xi$, so the cycle class of $\rho_*\eta\in A^k(\bcM_{g,(0^+)^n})_\QQ$ is $\xi$.

Conversely, suppose that the Hodge conjecture holds for $\bcM_{g,(0^+)^m}$ for all $m\le n$. Recall that we have the blowup decomposition for Chow groups \cite[\S6.7]{Fulton} as well as the ordinary cohomology groups \cite[p606]{GrHa} and the cycle class map preserves the decomposition. By induction on $r$ and $s$ with the blowup construction in Theorem \ref{9}, we find that the Hodge conjecture holds for $\bcM_{g,r|s}$ whenever $r+s\le n$. In particular, it holds for any $\bcM_{g,m}=\bcM_{g,m|0}$ with $m\le n$. The second statement follows from the same induction.
\end{proof}
\begin{coro}\label{HC}
    The Hodge conjecture holds for $\bcM_{1,r|s}$ for all $r,s$ and for $\bcM_{2,r|s}$ for $r+s\le 19$.
\end{coro}
\begin{proof}
    The even degree cohomology of $\bcM_{1,n}$ is tautological for all $n$ by \cite{Pet0}, and that of $\bcM_{2,n}$ is tautological for $n< 20$ by \cite{Pet2}. Since tautological classes are algebraic, the Hodge conjecture holds for these spaces. Now apply Proposition \ref{10}.
\end{proof}

\subsection{The equivariant change of variables}\label{S3.3}
The blowup centers in Theorem \ref{9} are not preserved by $\bS_n$, so the argument above does not give the $\bS_n$-equivariant statement. We use instead the stratification by rational tails.

For $g\ge 2$ and $m\ge 0$, let $\cM^{\circ}_{g,m}\subset\bcM_{g,m}$ be the open substack of stable curves without rational tails, i.e.\ without a connected subcurve of arithmetic genus zero meeting the rest of the curve in exactly one point. By the description of $\bcM_{g,(0^+)^m}$ above (cf.\ \cite[Lemma 2.7]{KSY}), $\cM^\circ_{g,m}$ is also the open substack of $\bcM_{g,(0^+)^m}$ where the $m$ points are pairwise distinct. Let $\cU_g$ be the $\bS$-space $m\mapsto\cM^\circ_{g,m}$, let $\overline{\cT}$ be the $\bS$-space with $\overline{\cT}(0)=\emptyset$, $\overline{\cT}(1)=\mathrm{pt}$ and $\overline{\cT}(k)=\bcM_{0,k+1}$ for $k\ge2$, with $\bS_k$ permuting the first $k$ points, and let $\mathcal E_{+}$ be the $\bS$-space with $\mathcal E_{+}(0)=\emptyset$ and $\mathcal E_{+}(k)=\mathrm{pt}$ for $k\ge1$, so that $\HH^{\bS}_{\mathcal E_+}=\bh(1)-1$. We write $\cX_g$ and $\cY_{g}$ for the $\bS$-spaces $n\mapsto\bcM_{g,n}$ and $n\mapsto\bcM_{g,(0^+)^n}$
so that 
\beq
\HH^{\bS}_{\cX_g}=\sum_{n\ge 0}\HH^{\bS}_{\bcM_{g,n}}=:\cH^{\bS}_g \quad \text{and}\quad \HH^{\bS}_{\cY_g}=\sum_{n\ge 0}\HH^{\bS}_{\bcM_{g,(0^+)^n}}=:\cL^{\bS}_g
.\eeq

\begin{prop}\label{tails}
For $g\ge 2$, in the Grothendieck ring of $\bS$-spaces,
\[
[\cX_g]=[\cU_g\pl\overline{\cT}],\qquad [\cY_g]=[\cU_g\pl\mathcal E_{+}].
\]
\end{prop}
\begin{proof}
Every stable $n$-pointed curve of genus $g\ge2$ is obtained uniquely from its core, obtained by contracting all maximal rational tails to points, as follows: the maximal rational tails define a set partition of $[n]$, the singleton blocks being the marked points on the core, and the core with one marked point for each block is a point of $\cM^\circ_{g,m}$; each block $B$ with $|B|\ge2$ carries a rational tail, which is a point of $\bcM_{0,|B|+1}$. The gluing maps identify the locally closed stratum with a given block structure with the corresponding quotient of a product, compatibly with $\bS_n$, which proves the first identity (this is the argument of \cite[Proposition 4.6]{KSY} with $m=0$). For the second, stratify $\bcM_{g,(0^+)^n}$ by the partition of $[n]$ into blocks of coincident points; the stratum with $m$ blocks is identified with $\cM^\circ_{g,m}$ with the blocks as labels (cf.\ \cite[Proposition 4.7]{KSY}).
\end{proof}

Let $\overline{\mathsf b}=\HH^{\bS}_{\overline{\cT}}=h_1+\sum_{k\ge2}\HH^{\bS}_{\bcM_{0,k+1}}$ and $\mathsf m=\sum_{k\ge2}\HH^{\bS}_{\cM_{0,k+1}}$, where in both cases $\bS_k$ fixes the last point.

\begin{lemm}\label{Zclosed}
\begin{enumerate}[label=(\alph*)]
\item $\overline{\mathsf b}=h_1+\mathsf m\pl\overline{\mathsf b}$, so $\overline{\mathsf b}$ is the plethystic inverse of $h_1-\mathsf m$.
\item $h_1-\mathsf m=\dfrac{x^2h_1-\mathsf F+1}{x(x-1)}$ where $\mathsf F=\sum_{k\ge0}\HH^{\bS}_{F(\CC,k)}=\prod_{k\ge1}(1+p_k)^{\frac1k\sum_{d|k}\mu(k/d)x^d}$ is the series of the configuration spaces $F(\CC,k)$ of $k$ distinct points in $\CC$.
\item $(h_1-\mathsf m)\pl(\bh(1)-1)=Z$, with $Z$ in \eqref{defZ}.
\end{enumerate}
\end{lemm}
\begin{proof}
(a) Stratify $\bcM_{0,k+1}$ by the open stratum $\cM_{0,j+1}$ of the component containing the last point. Each of its other $j\ge2$ special points is either a marked point or a node to a subtree, which is a point of $\bcM_{0,|B|+1}$ with $B$ the set of labels on it. This is the identity $[\overline{\cT}]=[\mathcal E_1]+[\mathcal M\pl\overline{\cT}]$ of $\bS$-spaces, where $\mathcal E_1$ is a point in degree $1$ and $\mathcal M$ is $k\mapsto\cM_{0,k+1}$ for $k\ge2$, cf.\ \cite{Get95a}. By \eqref{lambdaring}, $(h_1-\mathsf m)\pl\overline{\mathsf b}=h_1$; since both series begin with $h_1$, $\overline{\mathsf b}$ is a two-sided plethystic inverse.

(b) The formula for $\mathsf F$ is Getzler's \cite{Get95b}, since $\HH_{\CC}=x$. For $k\ge2$, the morphism $F(\CC,k)\to\cM_{0,k+1}$ sending a configuration to $(\PP^1,z_1,\dots,z_k,\infty)$ is a principal bundle under the affine group $\mathrm{Aff}_1\cong\CC\rtimes\CC^*$, which is trivial (normalize $z_1=0$, $z_2=1$), connected and commutes with $\bS_k$. Hence $\HH^{\sigma}_{F(\CC,k)}=x(x-1)\HH^{\sigma}_{\cM_{0,k+1}}$ for all $\sigma\in\bS_k$, by the Leray spectral sequence, whose $E_2$-term is $H^*_c(\cM_{0,k+1})\otimes H^*_c(\mathrm{Aff}_1)$ with $\sigma$ acting trivially on the second factor. Since $\HH^{\bS}_{F(\CC,0)}=1$ and $\HH^{\bS}_{F(\CC,1)}=xh_1$, we get $\mathsf m=(\mathsf F-1-xh_1)/(x(x-1))$.

(c) We have $h_1\pl(\bh(1)-1)=\bh(1)-1$ and $(1+p_k)\pl(\bh(1)-1)=p_k\pl\bh(1)=\exp(\sum_m p_{km}/m)$. Since the exponents in $\mathsf F$ stay unchanged,
\[
\mathsf F\pl(\bh(1)-1)=\exp\Big(\sum_{N\ge1}\frac{p_N}{N}\sum_{k|N}\sum_{d|k}\mu(k/d)x^d\Big)=\exp\Big(\sum_N\frac{x^Np_N}{N}\Big)=\bh(x)
\]
by M\"obius inversion. Substituting into (b) gives the first expression for $Z$; the second follows from $\frac{x^2-x^k}{x(x-1)}=-\frac{x(x^{k-2}-1)}{x-1}$ for $k\ge 3$, while the coefficient of $h_1$ is $1$ and that of $h_2$ is $0$.
\end{proof}

\begin{theo}\label{A}
For $g\ge2$, $\cL^{\bS}_g=\cH^{\bS}_g\pl Z$, with $Z$ in \eqref{defZ}.
\end{theo}
\begin{proof}
By Proposition \ref{tails} and \eqref{lambdaring}, $\cH^{\bS}_g=\HH^{\bS}_{\cU_g}\pl\overline{\mathsf b}$ and $\cL^{\bS}_g=\HH^{\bS}_{\cU_g}\pl(\bh(1)-1)$. By Lemma \ref{Zclosed} (a), $\HH^{\bS}_{\cU_g}=\cH^{\bS}_g\pl(h_1-\mathsf m)$, so $\cL^{\bS}_g=\cH^{\bS}_g\pl(h_1-\mathsf m)\pl(\bh(1)-1)=\cH^{\bS}_g\pl Z$ by Lemma \ref{Zclosed} (c) and the associativity of plethysm.
\end{proof}

\begin{rema}\label{g1}
Proposition \ref{tails}, Theorem \ref{A} and Proposition \ref{7} also hold for $g=1$,
with the convention that the series $\cH^{\bS}_1$, $\cL^{\bS}_1$, $\cH_1$ and $\cL_1$
start with $n=1$, since $\bcM_{1,0}$ and $\bcM_{1,(0^+)^0}$ do not exist. 
The only point in the proof of Proposition \ref{tails} where $g\ge2$ is used is that every stable curve has a nonempty core. For $g=1$ and $n\ge1$, the core of a stable $n$-pointed curve is a smooth elliptic curve or a cycle of rational curves carrying at least one block of marked points, hence a point of $\cM^\circ_{1,m}$ with $m\ge1$. Likewise, the locus of
$\bcM_{1,(0^+)^n}$ where the marked points are pairwise distinct is $\cM^\circ_{1,n}$.
Hence $[\cX_1]=[\cU_1\pl\overline{\cT}]$ and $[\cY_{1}]=[\cU_1\pl\mathcal E_+]$
with $\cU_1(0)=\emptyset$, $\cX_1(0)=\emptyset$ and $\cY_1(0)=\emptyset$. Since Lemma \ref{Zclosed} does not involve $g$, the proof of
Theorem \ref{A} applies verbatim and gives $\cL^{\bS}_1=\cH^{\bS}_1\pl Z$. Taking the rank
gives $\cL_1(y)=\cH_1(z(y))$, which also follows from Theorem \ref{9}, as Theorem \ref{9}
holds for $g\ge1$. 
\end{rema}

\begin{rema}\label{Arem}
(1) Since $Z$ has no constant term and $\rk(h_k)=y^k/k!$, Lemma \ref{rankmap} shows that the image of Theorem \ref{A} under the rank map is $\cL_g(y)=\cH_g(z(y))$ with $z(y)=\rk(Z)$ as in \eqref{5}. This gives a second, blowup-free proof of Proposition \ref{7}. Likewise, Lemma \ref{Zclosed} becomes $z=\frac{x^2\varphi-(1+\varphi)^x+1}{x(x-1)}$ with $\varphi=\rk(\bh(1)-1)=e^y-1$, and $\rk(\overline{\mathsf b})=\varphi(z)$, which explains Remark \ref{Manin}.

(2) Computationally, $Z^{\langle-1\rangle}$ is obtained from the fixed point iteration $W\mapsto h_1+(h_1-Z)\pl W$, which converges degree by degree since $h_1-Z$ starts in degree $3$, and $\overline{\mathsf b}=\exp\big(\sum_k\frac1k p_k\pl Z^{\langle-1\rangle}\big)-1$ recovers the $\bS_k$-equivariant Hodge polynomials of $\bcM_{0,k+1}$ \cite{Get95a}. For instance, $\HH^{\bS}_{\bcM_{0,5}}=(1+2x+x^2)h_4+x\,s_{31}$.
\end{rema}

\bigskip
\section{Moduli stacks of minimally weighted pointed curves}\label{S4}
In this section, for $g\ge 2$, we stratify $\bcM_g$ by using stable graphs $\Gamma$ and then consider the inverse images of the strata $\cM_\Gamma$ by the forgetful morphism
\beq\label{17}
\Phi_{g,n}: \bcM_{g,(0^+)^n}=\bcM_{g,0|n}\lra \bcM_g .\eeq
Adding up the twisted characteristics of $\Phi_{g,n}^{-1}(\cM_\Gamma)$, we obtain a formula for $\cL^{\bS}_g$ (Theorem \ref{graphsum}) and its numerical form (Corollary \ref{24}).

\subsection{Stable graphs and the strata of $\bcM_{g,(0^+)^n}$}
A graph $\Gamma$ equipped with an integer $g_v\ge 0$ for each vertex $v\in V(\Gamma)$ is called \emph{stable} if $2g_v-2+d_v>0$ holds for each vertex, where $d_v$ denotes the valency at the vertex $v$. We regard an edge as a pair of half-edges, so that $\Aut(\Gamma)$ acts on the set of half-edges. In particular, a loop at $v$ contributes two half-edges at $v$, and $d_v$ is the number of half-edges at $v$. The genus of $\Gamma$ is
$$g=\frac12\sum_v (2g_v-2+d_v)+1.$$
Let $\cG_g$ be the set of all stable graphs of genus $g$. It is well known that we have the stratification
\beq\label{18}
\bcM_g=\bigsqcup_{\Gamma\in \cG_g} \cM_\Gamma, \quad \cM_\Gamma=\prod_{v\in V(\Gamma)}\cM_{g_v,d_v}/\Aut(\Gamma).
\eeq

When $\Gamma=\varnothing$ is the smooth graph, i.e.\ the graph with only one vertex of genus $g$ and no edges, $\cM_\varnothing=\cM_g$ and $\Phi_{g,n}^{-1}(\cM_g)$ is the $n$-fold fiber product
$$\cC_g^{(n)}:=\cC_g\times_{\cM_g}\cdots\times_{\cM_g}\cC_g$$
of the universal curve $q_g:\cC_g\to \cM_g$, with $\bS_n$ permuting the factors. By the relative K\"unneth theorem, we have
\beq\label{19}
R(\Phi_{g,n}|_{\Phi_{g,n}^{-1}(\cM_g)})_*\QQ=(R{q_g}_*\QQ)^{\otimes n}=\left(\QQ\oplus \QQ(-1)[-2]\oplus \V_g[-1]\right)^{\otimes n},\quad \V_g:=R^1{q_g}_*\QQ.\eeq

For the boundary strata, there are nodes and we may insert the Losev-Manin spaces
$$L_k:=\bcM_{0,2|k}$$ at nodes. We let $L_0$ be a point. 
For a stable graph $\Gamma\in \cG_g$ and a vertex $v\in V(\Gamma)$, we let
$$\mathring{\cC}_v:=\cM_{g_v,d_v+1}\lra \cM_{g_v,d_v}=:\cM_v$$
be the universal curve with the $d_v$ universal sections deleted, and
\[ \mathring{\cC}_v^{(k)}=\mathring{\cC}_v\times_{\cM_v}\cdots\times_{\cM_v}\mathring{\cC}_v\]
its $k$-fold fiber product. Let $E(\Gamma)$ denote the set of edges in $\Gamma$. Then we have
\beq\label{50}
\Phi_{g,n}^{-1}(\cM_\Gamma)= \Big(\bigsqcup_{I_\cdot}\ \prod_{v\in V(\Gamma)}\mathring{\cC}_v^{(|I_v|)}\times \prod_{e\in E(\Gamma)}L_{|I_e|}\Big)\Big/\Aut(\Gamma)\eeq
where the disjoint union runs over all decompositions $\{1,2,\cdots,n\}=(\bigsqcup_{v\in V(\Gamma)}I_v)\sqcup (\bigsqcup_{e\in E(\Gamma)}I_e)$, which are permuted by $\Aut(\Gamma)$. Indeed, the light points on the component corresponding to $v$ form a point of $\mathring{\cC}_v^{(|I_v|)}$ since they may coincide with each other but not with the nodes, and the light points which approach the node corresponding to $e$ lie on a chain of rational bridges inserted at the node, parametrized by $L_{|I_e|}$. The decomposition \eqref{50} is compatible with the action of $\bS_n$ on the labels. In the language of \S\ref{S2.2}, the disjoint union over the decompositions $I_\cdot$ is a species product, so that we have an isomorphism of $\bS$-spaces
\beq\label{speciesprod}
\Phi^{-1}_{g,\bullet}(\cM_\Gamma)=\Big(\prod_{v\in V(\Gamma)}\mathring\cC_v^{(\bullet)}\cdot\prod_{e\in E(\Gamma)}L_\bullet\Big)\Big/\Aut(\Gamma),
\eeq
where $\mathring\cC^{(\bullet)}_v$ is the $\bS$-space $r\mapsto\mathring\cC^{(r)}_v$ and $L_\bullet$ is the $\bS$-space $k\mapsto L_k$. An automorphism $\alpha\in\Aut(\Gamma)$ maps the factor of $v$ to the factor of $\alpha v$, through the bijection between the half-edges at $v$ and at $\alpha v$, and the factor of $e$ to the factor of $\alpha e$, through the flip of the chains if $\alpha$ reverses the orientation of $e$ relative to chosen orientations.

\begin{defi}\label{Adef}
For $\beta\in\bS_d$, which acts on $\cM_{g,d}$ and on the punctured universal curve $\mathring\cC_{g,d}=\cM_{g,d+1}\to\cM_{g,d}$ by relabelling the first $d$ points, let
\beq\label{defA}
A^{\bS,\beta}_{g,d}=\HH^{\bS,\beta}_{\mathring\cC^{(\bullet)}_{g,d}}\in\Lam,\qquad A^{\bS}_{g,d}=A^{\bS,\id}_{g,d},\qquad A^\beta_{g,d}(y)=\rk(A^{\bS,\beta}_{g,d})=\sum_{r\ge0}\HH^\beta_{\mathring\cC^{(r)}_{g,d}}\frac{y^r}{r!}.\eeq 
These depend only on the conjugacy class of $\beta$. Let $\iota$ be the flip of the chains, exchanging the two heavy points of $L_k=\bcM_{0,2|k}$, and let
\[C^{\bS}=\HH^{\bS}_{L_\bullet},\quad D^{\bS}=\HH^{\bS,\iota}_{L_\bullet},\qquad C(y)=\rk(C^{\bS})=\sum_{k\ge0}\HH_{L_k}\frac{y^k}{k!},\quad D(y)=\rk(D^{\bS})=\sum_{k\ge0}\HH^\iota_{L_k}\frac{y^k}{k!}.\]
\end{defi}

\subsection{The graph sum}
For $\alpha\in\Aut(\Gamma)$ and a vertex $v$ whose $\langle\alpha\rangle$-orbit has length $\ell$, let $\alpha^\ell(v)\in\bS_{d_v}$ be the permutation of the half-edges at $v$ induced by $\alpha^\ell$. For an edge $e$ whose orbit has length $\ell$, $\alpha^\ell$ either preserves or flips $e$.

\begin{theo}\label{graphsum}
For $g\ge2$,
\beq\label{23S}
\cL^{\bS}_g=\sum_{\Gamma\in\cG_g}\cL^{\bS}_\Gamma,\qquad
\cL^{\bS}_\Gamma=\frac1{|\Aut\Gamma|}\sum_{\alpha\in\Aut(\Gamma)}\ \prod_{O}p_{\ell(O)}\pl A^{\bS,\alpha^{\ell(O)}(v_O)}_{g_{v_O},d_{v_O}}\ \prod_{P}p_{\ell(P)}\pl B^\alpha_P,
\eeq
where $\cL^{\bS}_\Gamma=\sum_n\HH^{\bS}_{\Phi^{-1}_{g,n}(\cM_\Gamma)}$, $B^\alpha_P=C^{\bS}$ if $\alpha^{\ell(P)}$ preserves the ends of $e_P$ and $B^\alpha_P=D^{\bS}$ if $\alpha^{\ell(P)}$ flips $e_P$, and the products run over the $\langle\alpha\rangle$-orbits $O$ of vertices and $P$ of edges, of lengths $\ell(O)$ and $\ell(P)$, with representatives $v_O\in O$ and $e_P\in P$.
\end{theo}
\begin{proof}
Since the stratification $\bcM_{g,(0^+)^n}=\bigsqcup_\Gamma\Phi^{-1}_{g,n}(\cM_\Gamma)$ is $\bS_n$-stable, $\cL^{\bS}_g=\sum_\Gamma\cL^{\bS}_\Gamma$. By \eqref{speciesprod} and \eqref{Burnside}, $\cL^{\bS}_\Gamma=\frac1{|\Aut\Gamma|}\sum_\alpha\HH^{\bS,\alpha}_{\cX_\Gamma}$ where $\cX_\Gamma$ is the species product in \eqref{speciesprod}. Group the factors of $\cX_\Gamma$ according to the $\langle\alpha\rangle$-orbits of vertices and edges. Then $\cX_\Gamma$ is the species product of $\alpha$-stable pieces, one for each orbit, and the twisted characteristic of a species product of $\alpha$-stable pieces is the product of their twisted characteristics. By Lemma \ref{strata}, only the decompositions $[n]=I\sqcup J$ with $\sigma(I)=I$ contribute to the trace of $\sigma\alpha$, and then the trace is the product by the K\"unneth formula, exactly as for \eqref{lambdaring}.

For an orbit $O=\{v,\alpha v,\dots,\alpha^{\ell-1}v\}$ of vertices, identify the factor of $\alpha^iv$ with $\mathring\cC_v^{(\bullet)}$ through $\alpha^i$. Then $\alpha$ maps the $i$-th factor to the $(i+1)$-st by the identity for $i<\ell-1$ and the last factor to the first by $\alpha^\ell(v)$, acting by relabelling the sections. By Lemma \ref{orbit}, the piece of $O$ contributes $p_\ell\pl\HH^{\bS,\alpha^\ell(v)}_{\mathring\cC_v^{(\bullet)}}=p_\ell\pl A^{\bS,\alpha^\ell(v)}_{g_v,d_v}$. Similarly, an orbit of edges of length $\ell$ contributes $p_\ell\pl C^{\bS}$ or $p_\ell\pl D^{\bS}$ according to whether $\alpha^\ell$ preserves or flips $e_P$.
\end{proof}

\begin{rema}\label{KSrem}
Theorem \ref{graphsum} can also be derived from the formalism of Kannan and Song \cite{KS}. In their notation, let $\mathsf a_h=\sum_n\HH^{\bS}_{\cM_{h,n}}$, let $\overline\partial_\mu$ denote skewing by $p_\mu$ (up to the factor $\prod_i i^{\mu_i}$),
and let
\[\mathsf c_1=\frac1{1-\mathsf a_0''},\qquad
  \mathsf c_2=\frac{1+2\dot{\mathsf a}_0}{1-\psi_2(\mathsf a_0'')}.\]
By \cite[Lemma 5.13 and the proof of Theorem A]{KS}, the series $\HH^{\bS}_{\cU_g}$ of stable curves without rational tails is a sum over stable graphs of the same shape as \eqref{23S}, with vertex factors $\overline\partial_\mu\mathsf a_h$ and edge factors $\mathsf c_1$, $\mathsf c_2$. Their Theorem A is obtained by composing with 
$p_1+\partial\overline{\mathsf a}_0/\partial p_1=\overline{\mathsf b}$, in accordance with the first identity of Proposition \ref{tails}. By the second identity,
$\cL^{\bS}_g=\HH^{\bS}_{\cU_g}\pl(\bh(1)-1)$. Since $f\mapsto f\pl(\bh(1)-1)$ is a ring homomorphism which commutes with the operations $p_k\pl(-)$, it may be applied term by term. Stratifying by coincidences of the light points, as in the proof of Proposition \ref{tails}, gives
\[
(\overline\partial_\mu\mathsf a_h)\pl(\bh(1)-1)=A^\beta_{h,|\mu|},\qquad
\mathsf c_1\pl(\bh(1)-1)=C^{\bS},\qquad
\mathsf c_2\pl(\bh(1)-1)=D^{\bS},
\]
where $\beta$ has cycle type $\mu$. For instance, the second identity also follows from Getzler's formula for $\HH^{\bS}_{F(\CC^*,k)}$ by M\"obius inversion, as in Lemma \ref{Zclosed}(c). This recovers Theorem \ref{graphsum}: it is the version of \cite[Theorem A]{KS} in which the plethysm with $\overline{\mathsf b}$ is replaced by plethysm with $\bh(1)-1$. Conversely, our genus two computation below supplies the input
$\mathsf a_2$ of \cite[Theorem A]{KS} for all $n$. Indeed
$A_{2,0}=\mathsf a_2\pl(\bh(1)-1)$, so Proposition \ref{Agen1} gives
\[
\mathsf a_2=\big(\bh(1)\bh(x)\cF^{\bS}\big)\pl\sum_{k\ge1}\frac{\mu(k)}k\log(1+p_k),
\]
the last series being the plethystic inverse of $\bh(1)-1$.
\end{rema}

Taking the rank, we obtain the numerical form. For $\alpha\in \Aut(\Gamma)$, let $\kappa_+(\alpha)$ (resp.\ $\kappa_-(\alpha)$) be the number of $\alpha$-fixed edges $e$ in $\Gamma$ whose two half-edges are fixed (resp.\ interchanged) by $\alpha$. 
Let $\cL_\Gamma=\rk(\cL^{\bS}_\Gamma)=\sum_n\HH_{\Phi^{-1}_{g,n}(\cM_\Gamma)}\frac{y^n}{n!}$.

\begin{coro}\label{24} For $g\ge 2$,
    \beq\label{23}
    \cL_g(y)=\sum_{\Gamma\in \cG_g}\cL_\Gamma,\quad \cL_\Gamma=\frac1{|\Aut(\Gamma)|}\sum_{\alpha\in \Aut(\Gamma)}
    \prod_{O}\psi^{\ell(O)}\HH^{\alpha^{\ell(O)}(v_O)}_{\cM_{v_O}}\;
    \prod_{\alpha v=v}A^{\alpha(v)}_{g_v,d_v}(y)\;
    C(y)^{\kappa_+(\alpha)}D(y)^{\kappa_-(\alpha)},\eeq
    where $O$ runs over the orbits of $\langle\alpha\rangle$ on $V(\Gamma)$ of length $\ell(O)\ge 2$ and $v_O\in O$.
\end{coro}
\begin{proof}
Apply $\rk$ to \eqref{23S}. By Lemma \ref{rankmap}, a factor $p_\ell\pl f$ with $\ell\ge2$ contributes $\psi^\ell(f(0))$. The constant term of $A^{\bS,\beta}_{g,d}$ is $\HH^\beta_{\cM_{g,d}}$ (the case $r=0$) and those of $C^{\bS}$ and $D^{\bS}$ are $\HH_{L_0}=1$.
\end{proof}

Corollary \ref{24} can also be proved directly from \eqref{50}: by Lemma \ref{15}, a decomposition $I_\cdot$ contributes to $\HH^\alpha$ only if $I_v=\emptyset$ whenever $\alpha v\ne v$ and $I_e=\emptyset$ whenever $\alpha e\ne e$, and the trace of a cyclic permutation of the factors $\cM_w$, $w\in O$, is $\psi^{\ell(O)}\HH^{\alpha^{\ell(O)}(v_O)}_{\cM_{v_O}}$ if $\ell(O)\ge 2$.

\begin{rema}\label{63}
For $g=2$, every vertex moved by an automorphism of a stable graph has $(g_v,d_v)=(1,1)$ or $(0,3)$, and the orbits have length $2$ (see \S\ref{S5.1}). Since $\HH_{\cM_{1,1}}=x$ and $\HH_{\cM_{0,3}}=1$, the factor in \eqref{23} for such an orbit is $\psi^2(x)=x^2$ or $1$, which is the product of the Hodge-Deligne polynomials of the two vertex moduli spaces. For general $g$, however, $\psi^{\ell}\HH_{\cM_v}$ differs from $\HH_{\cM_v}^\ell$. 
\end{rema}

\begin{exam}\label{39}
For the smooth graph, $\cL^{\bS}_\varnothing=A_{g,0}$. By \eqref{19} and Lemma \ref{fibered},
\beq\label{29}
\cL^{\bS}_\varnothing=\bh(1)\bh(x)\cF^{\bS}_g,\qquad \cF^{\bS}_g=\sum_{\lambda}(-1)^{|\lambda|}s_{\lambda'}\HH_{\cM_g;\mathbb S_\lambda\V_g},
\eeq
and numerically $\cL_\varnothing=e^{(1+x)y}\cF_g(y)$ with $\cF_g(y)=\sum_{k\ge 0}\frac{y^k}{k!}\HH_{\cM_g;\V_g[-1]^{\otimes k}}$.
In \S\ref{S5}, we will compute $\cF_2^{\bS}$ and $\cF_2$.
\end{exam}

\subsection{Losev-Manin spaces}
The Losev-Manin series $C^{\bS}$ and $D^{\bS}$ are explicitly computed as follows.

\begin{prop}\label{LM}
\[C^{\bS}=\frac{x-1}{x-\bh(x)\be(-1)},\qquad D^{\bS}=\frac{\bh(x)\be(1)+x}{x+1}\cdot(p_2\pl C^{\bS}).\]
\end{prop}
\begin{proof}
Stratify $L_k$ by the number $j\ge1$ of components of the chain together with the ordered partition $S_1\sqcup\dots\sqcup S_j=\{1,\dots,k\}$ recording which labels lie on which component, the components being ordered from the $0$-end to the $\infty$-end. By stability, every $S_i$ is nonempty. Each component of the chain carries exactly two special points, 
and so its automorphism group, as a $\PP^1$ with two marked points, is $\C^*$. The moduli of the $|S_i|\ge1$ labels on it is $Q_{|S_i|}=(\C^*)^{|S_i|}/\C^*$. The chain is linearly ordered and its two ends are distinguishable, so there are no further automorphisms, and the stratum is $\prod_iQ_{|S_i|}$. Hence $L_\bullet=\sum_{j\ge0}Q^{\cdot j}$ in the Grothendieck ring of $\bS$-spaces, with $Q$ as in Lemma \ref{Cstar}, and $C^{\bS}=(1-\HH^{\bS}_Q)^{-1}$, which is the first formula by Lemma \ref{Cstar}.

The flip $\iota$ maps the stratum indexed by $(S_1,\dots,S_j)$ to the one indexed by $(S_j,\dots,S_1)$, acting by inversion on each component. For $\sigma\in\bS_k$, $\sigma\iota$ preserves the stratum if and only if $\sigma(S_i)=S_{j+1-i}$ for all $i$. In terms of $\bS$-spaces, $\iota$ acts on $Q^{\cdot j}$ by reversing the order of the factors, composed with the inversion on each factor. The factors $i$ and $j+1-i$ with $i<j+1-i$ form an orbit of length $2$. Identifying the factor $j+1-i$ with the factor $i$ through the inversion, $\iota$ maps the first to the second by the identity and the second to the first by the composite of two inversions, which is the identity. By Lemma \ref{orbit}, each such pair contributes $p_2\pl\HH^{\bS}_Q$. If $j$ is odd, the middle factor is preserved, with $\iota$ acting by the inversion, and contributes $\HH^{\bS,\iota}_Q$. Hence
\[
D^{\bS}=\big(1+\HH^{\bS,\iota}_Q\big)\sum_{m\ge0}\big(p_2\pl\HH^{\bS}_Q\big)^m=\big(1+\HH^{\bS,\iota}_Q\big)\,(p_2\pl C^{\bS}),
\]
and Lemma \ref{Cstar} gives $1+\HH^{\bS,\iota}_Q=(\bh(x)\be(1)+x)/(x+1)$.
\end{proof}

\begin{coro}\label{25}
For $k\ge1$,
\beq\label{27}
\HH_{L_k}=\sum_{j=1}^k j!\genfrac\{\}{0pt}{}{k}{j}(x-1)^{k-j},
\qquad
\HH^\iota_{L_k}=(x+1)^{k-1},
\eeq
where $\genfrac\{\}{0pt}{}{k}{j}$ is a Stirling number of the second kind. Summing them up, we get
\beq\label{26}
C(y)=\frac{x-1}{x-e^{(x-1)y}},\qquad
D(y)=\frac{e^{(x+1)y}+x}{x+1}.
\eeq
\end{coro}
\begin{proof}
The formulas \eqref{26} are the images of Proposition \ref{LM} under the rank map, since $\rk(\bh(t))=e^{ty}$, $\rk(\be(t))=e^{ty}$ and $\rk(p_2\pl C^{\bS})=1$ by Lemma \ref{rankmap}. The second formula in \eqref{27} is the coefficient of $y^k/k!$ in $D(y)$. For the first, the strata in the proof of Proposition \ref{LM} with $j$ components are indexed by the surjections $[k]\twoheadrightarrow[j]$, of which there are $j!\genfrac\{\}{0pt}{}{k}{j}$, and each has class $\prod_{i}(x-1)^{|S_i|-1}=(x-1)^{k-j}$. Alternatively, $\iota$ preserves no stratum with $j\ge2$, since $S_i=S_{j+1-i}$ is impossible for disjoint nonempty sets, so by Lemma \ref{15} only $j=1$ contributes to $\HH^\iota_{L_k}$, and Lemma \ref{16} gives $(x+1)^{k-1}$.
\end{proof}

\begin{exam}
$L_1=\operatorname{pt}$; $L_2\cong\PP^1$, with $\HH=x+1$; $L_3$ is the toric surface of the hexagon, with $\HH=x^2+4x+1$. In general $L_k$ is the permutohedral toric variety of dimension $k-1$, and $\HH_{L_k}$ is its Eulerian polynomial. Equivariantly, $\HH^{\bS}_{L_2}=(1+x)h_2$ and $\HH^{\bS}_{L_3}=(1+2x+x^2)h_3+x\,s_{21}$. The twisted characteristic $D^{\bS}$ records the traces of $\sigma\iota$. For instance $\HH^{\bS,\iota}_{L_2}=(1+x)h_2$ and
\[\HH^{\bS,\iota}_{L_3}=(x+1)^2\frac{p_1^3}6+(x^2+1)\frac{p_1p_2}2+(x^2-x+1)\frac{p_3}3=(1+x^2)h_3+x\,s_{21},\]
so that the traces of $\iota$, $(12)\iota$ and $(123)\iota$ on $H^*(L_3)$ are $(x+1)^2$, $x^2+1$ and $x^2-x+1$.
\end{exam}

\bigskip
\section{The genus two case: local systems and punctured curves}\label{S5}
In this section, we specialize to $g=2$ and compute all the twisted characteristics $A^{\bS,\beta}_{g,d}$ which appear in Theorem \ref{graphsum}.

\subsection{Stable graphs of genus 2}\label{S5.1}
\begin{prop}\label{28}
There are exactly seven stable graphs of genus $2$ as follows.
\begin{center}
\footnotesize
\begin{tabular}{@{}llccl@{}}
\toprule
 & description & vertex types $(g_v,d_v)$ & codim & $\Aut(\Gamma)$\\
\midrule
$\varnothing$ & one genus-$2$ vertex, no edges & $(2,0)$ & $0$ & trivial\\
$\Gamma_1$ & two genus-$1$ vertices joined by one edge & $(1,1),(1,1)$ & $1$ & $\Z/2$, swap\\
$\Gamma_0$ & one genus-$1$ vertex with a loop & $(1,2)$ & $1$ & $\Z/2$, flip\\
$\Gamma_a$ & genus-$0$ vertex with a loop, joined to a genus-$1$ vertex & $(0,3),(1,1)$ & $2$ & $\Z/2$, flip\\
$\Gamma_b$ & one genus-$0$ vertex with two loops & $(0,4)$ & $2$ & $D_4$, order $8$\\
$\Gamma_c$ & two genus-$0$ vertices, three parallel edges (theta) & $(0,3),(0,3)$ & $3$ & $\bS_3\times\Z/2$, order $12$\\
$\Gamma_d$ & two genus-$0$ vertices with a loop each, a bridge (dumbbell) & $(0,3),(0,3)$ & $3$ & $D_4$, order $8$\\
\bottomrule
\end{tabular}
\end{center}
\end{prop}
The proof is straightforward and we omit it. 

\subsection{Cohomology of local systems}\label{S5.2}
For $g,n\ge 0$, let $$q_{g,n}:\cC_{g,n}=\bcM_{g,n+1}\times_{\bcM_{g,n}}\cM_{g,n}\lra \cM_{g,n}$$ be the universal curve and
$$\mathring{q}_{g,n}:\mathring{\cC}_{g,n}=\cM_{g,n+1}\lra \cM_{g,n}$$
be $\cC_{g,n}$ minus the $n$ universal sections. When $n=0$, $\cC_g=\cC_{g,0}=\mathring{\cC}_{g,0}=\mathring{\cC}_{g}$. The fibers of $\mathring{q}_{1,1}$ are elliptic curves minus the identity, whose compactly supported cohomology is $\QQ(-1)[-2]\oplus \QQ^2[-1]$. Let $V=R^1{q_{1,1}}_*\QQ$, a local system of rank 2 over $\cM_{1,1}$, and let $\V=\V_2=R^1{q_2}_* \QQ$, of rank 4 over $\cM_2$. Let
\beq\label{EFdef}
\cE^{\bS}=\sum_{\lambda}(-1)^{|\lambda|}s_{\lambda'}\HH_{\cM_{1,1};\mathbb S_\lambda V},\qquad \cF^{\bS}=\cF^{\bS}_2=\sum_{\lambda}(-1)^{|\lambda|}s_{\lambda'}\HH_{\cM_{2};\mathbb S_\lambda\V},
\eeq
and $\cE(y)=\rk(\cE^{\bS})$, $\cF(y)=\rk(\cF^{\bS})$. Since $\sum_{\lambda\vdash k}\dim V_{\lambda'}\,\mathbb S_\lambda W=W^{\otimes k}$,
\[\cE(y)=\sum_{k\ge 0}\frac{y^k}{k!}\HH_{\cM_{1,1};V[-1]^{\otimes k}}, \quad \cF(y)=\sum_{k\ge 0}\frac{y^k}{k!}\HH_{\cM_{2};\V[-1]^{\otimes k}}.\]

\begin{lemm}\label{35}
$\HH_{\mathcal M_{1,1};\mathbb S_\lambda V}=0$ and $\HH_{\mathcal M_2;\mathbb S_\lambda \V}=0$ for $|\lambda|$ odd. Therefore
\beq\label{55}\cE^{\bS}=\sum_{|\lambda|\ \mathrm{even}}s_{\lambda'}\,\HH_{\cM_{1,1};\mathbb S_\lambda V},\quad \cF^{\bS}=\sum_{|\lambda|\ \mathrm{even}}s_{\lambda'}\,\HH_{\cM_2;\mathbb S_\lambda\V},\eeq
and $\cE(y)$, $\cF(y)$ are even functions of $y$.
\end{lemm}
\begin{proof}
Every elliptic curve carries the automorphism $-1$, and every genus $2$ curve carries the hyperelliptic involution, which are generic stabilizers. Both act by $-1$ on $H^1$ of the curve, and by $(-1)^{|\lambda|}$ on $\mathbb S_\lambda$. Cohomology of a stack with coefficients in a local system is computed by the invariants under such a generic automorphism, so it vanishes for $|\lambda|$ odd.
\end{proof}

Let $s_k=\dim S_k(\mathrm{SL}_2(\ZZ))$ be the dimension of the space of cusp forms of weight $k$, i.e.
\beq\label{41}
s_k=\left\{\begin{matrix}  \lfloor \frac{k}{12}\rfloor -1 & \text{if } k\equiv 2 \text{ mod } 12,\ k\ge 4,\\
\lfloor \frac{k}{12}\rfloor & \text{if } k\not\equiv 2 \text{ mod } 12,\ k\ge 4, \\
0 & \text{if } k \text{ is odd or } k\le 2.\end{matrix}\right.
\eeq

\begin{prop}\label{40}
We have $\HH_{\cM_{1,1};\Sym^0V}=x$ and $\HH_{\cM_{1,1};\Sym^{2k}V}=-1-s_{2k+2}(u^{2k+1}+v^{2k+1})$ for $k\ge1$, and
\[\cE^{\bS}=\sum_{\substack{\lambda_1\ge\lambda_2\ge0\\ \lambda_1+\lambda_2\ \mathrm{even}}}x^{\lambda_2}\,\HH_{\cM_{1,1};\Sym^{\lambda_1-\lambda_2}V}\ s_{(\lambda_1,\lambda_2)'}.\]
Numerically \cite[Proposition 4.1]{Ki},
\beq\label{45} \HH_{\cM_{1,1};V^{\otimes 2j}}=c_{jj} x^{j+1} - \sum_{i=0}^{j-1} c_{ij} x^i\left(1+s_{2j-2i+2}(u^{2j-2i+1}+v^{2j-2i+1})\right),\qquad c_{ij}=\binom{2j}{i}-\binom{2j}{i-1},
\eeq
where $c_{ij}$ is the multiplicity of $\Sym^{2j-2i}V(-i)$ in $V^{\otimes 2j}$.
\end{prop}
\begin{proof}
The first statement is the Eichler-Shimura theory: $\cM_{1,1}$ has coarse space $\mathbb A^1$, and for $k\ge1$, $H^*_c(\cM_{1,1};\Sym^{2k}V)$ is concentrated in degree $1$, where it is $\mathsf S_{2k+2}\oplus\QQ$, the second summand coming from the Eisenstein series; see \cite[\S4]{Ki}. Since $V$ has rank $2$, $\mathbb S_\lambda V=0$ unless $\lambda$ has at most two rows, and $\mathbb S_{(\lambda_1,\lambda_2)}V=\Sym^{\lambda_1-\lambda_2}V\otimes\det(V)^{\lambda_2}=\Sym^{\lambda_1-\lambda_2}V(-\lambda_2)$, which gives the formula for $\cE^{\bS}$ by \eqref{55}. Applying the rank map, the coefficient of $y^{2j}/(2j)!$ is $\sum_\lambda\dim V_{\lambda'}\HH_{\cM_{1,1};\mathbb S_\lambda V}$, and $\dim V_{(2j-i,i)'}=\dim V_{(2j-i,i)}=c_{ij}$ by the hook length formula, which gives \eqref{45}.
\end{proof}

The computation of $\cF^{\bS}$ requires Petersen's theorem which determines the cohomology of local systems on the moduli space $\sA_2$ of principally polarized abelian surfaces. Let $s_{j,k}=\dim S_{j,k}(\mathrm{Sp}_4(\Z))$ be the dimension of the space of vector-valued Siegel cusp forms of type $\Sym^j\otimes\det^k$, let $\mathsf S_k$ and $\mathsf S_{j,k}$ be the attached motives, of ranks $2s_k$ and $4s_{j,k}$ and of Hodge types $\{(k-1,0),(0,k-1)\}$ and $\{(j+2k-3,0),(j+k-1,k-2),(k-2,j+k-1),(0,j+2k-3)\}$, and let $\Sbar_{j,k}=\operatorname{gr}^W_{j+2k-3}\mathsf S_{j,k}$.

\begin{theo}\cite[Theorem~2.1]{Pet}\label{42}
Let $a\ge b\ge0$ with $a+b$ even and $(a,b)\ne(0,0)$. If we denote by $\V_{a,b}$ the local system on $\sA_2$ induced by the irreducible representation of $\mathrm{Sp}_4$ with highest weight $(a,b)$, then $H^k_c(\sA_2,\V_{a,b})=0$ for $k\notin\{2,3,4\}$, and up to semisimplification
\begin{align*}
H^4_c&=\begin{cases}s_{a+b+4}\Q(-b-2)&a=b\ \text{even}\\0&\text{else},\end{cases}\\
H^3_c&=\Sbar_{a-b,b+3}+s_{a+b+4}\mathsf S_{a-b+2}(-b-1)+\mathsf S_{a+3}
+\begin{cases}s'_{a+b+4}\Q(-b-1)&a=b\ \text{even}\\ s_{a+b+4}\Q(-b-1)&\text{else}\end{cases}\\
&\qquad+\begin{cases}\Q&a=b\ \text{odd}\\0&\text{else}\end{cases}
+\begin{cases}\Q(-1)&b=0\\0&\text{else},\end{cases}\\
H^2_c&=\mathsf S_{b+2}+s_{a-b+2}\Q
+\begin{cases}s'_{a+b+4}\Q(-b-1)&a=b\ \text{even}\\0&\text{else}\end{cases}
+\begin{cases}\Q&a>b>0,\ a,b\ \text{even}\\0&\text{else}.\end{cases}
\end{align*}
Moreover, the Hodge-Deligne polynomial of $\Sbar_{j,k}$ is
\[
\HH_{\Sbar_{j,k}}=
\begin{cases}
s_{j,k}\; W_{j,k}(u,v), & j>0\ \text{or}\ k\ \text{odd},\\[3pt]
(s_{0,k}-s_{2k-2})\;W_{0,k}(u,v)\;+\;s_{2k-2}\;\big(u^{2k-3}+v^{2k-3}\big),
& j=0,\ k\ \text{even}
\end{cases}
\]
where
\[W_{j,k}(u,v)= u^{j+2k-3}+u^{j+k-1}v^{k-2}+u^{k-2}v^{j+k-1}+v^{j+2k-3}.\]
\end{theo}
Here $s'_{a+b+4}$ is a certain integer defined in \cite{Pet}, which is conjecturally zero. Since it occurs in $H^2_c$ and $H^3_c$ with the same Tate twist, it cancels in the alternating sum, so that the Hodge-Deligne polynomials below do not depend on it. The second case for $\HH_{\Sbar_{0,k}}$ reflects the Saito-Kurokawa lifts, whose motives are $\mathsf S_{2k-2}+\QQ(-k+1)+\QQ(-k+2)$ and whose weight $2k-3$ part is $\mathsf S_{2k-2}$.

Taking the alternating sum $\HH=H^2_c-H^3_c+H^4_c$ gives the following.
\begin{coro}\label{43}
Let $a\ge b\ge0$ with $a+b$ even. Let $\eta_{a,b}=0$ if $a=b$ is even and $\eta_{a,b}=1$ otherwise, and for a property $P$, let $[P]$ be $1$ if $P$ holds and $0$ if not. Then $\HH_{\sA_2;\V_{0,0}}=x^3+x^2$, and for $(a,b)\ne(0,0)$,
\beq\label{46}
\begin{aligned}
\HH_{\sA_2;\V_{a,b}}\;=\;
&\ s_{b+2}\big(u^{b+1}+v^{b+1}\big)\;+\;s_{a-b+2}
\;+\;\big[\,a>b>0,\ a\equiv b\equiv0\ (2)\,\big]\\
&-\;s_{a+b+4}\,s_{a-b+2}\big(u^{a+2}v^{b+1}+u^{b+1}v^{a+2}\big)
\;-\;s_{a+3}\big(u^{a+2}+v^{a+2}\big)\\
&-\;\eta_{a,b}\,s_{a+b+4}\,x^{b+1}
\;-\;\big[\,a=b\ \mathrm{odd}\,\big]
\;-\;\big[\,b=0\,\big]\,x
\;+\;\big[\,a=b\ \mathrm{even}\,\big]\,s_{a+b+4}\,x^{b+2}\\
&-\begin{cases}
s_{a-b,b+3}\; W_{a-b,b+3}(u,v), & a>b\ \text{or}\ b\ \text{even},\\[3pt]
(s_{0,b+3}-s_{2b+4})\;W_{0,b+3}(u,v)\;+\;s_{2b+4}\;\big(u^{2b+3}+v^{2b+3}\big),
& a=b,\ b\ \text{odd}.
\end{cases}
\end{aligned}\eeq
\end{coro}
For example, $\HH_{\sA_2;\V_{1,1}}=-1$, $\HH_{\sA_2;\V_{2,0}}=-x$, $\HH_{\sA_2;\V_{2,2}}=0$, $\HH_{\sA_2;\V_{6,6}}=x^8$, and
\[
\HH_{\sA_2;\V_{11,1}}=1-x^2-u^{13}v^2-u^2v^{13}.
\]

For a partition $\lambda$ with at most four rows and $a\ge b\ge0$, let $m^\lambda_{a,b}$ be the multiplicity of $\V_{a,b}$ in the restriction of $\mathbb S_\lambda(\CC^4)$ from $\mathrm{GL}_4$ to $\mathrm{Sp}_4$. It may be computed by the Koike-Terada branching rule \cite{KT}, or directly by the Weyl character formula as the coefficient of $t_1^{a+2}t_2^{b+1}$ in $s_\lambda(t_1,t_1^{-1},t_2,t_2^{-1})\cdot\sum_{w\in W(C_2)}\mathrm{sgn}(w)\,t^{w(2,1)}$. For $j\ge0$, let $m^{2j}_{a,b}=\sum_{\lambda\vdash 2j}\dim V_\lambda\, m^\lambda_{a,b}$ be the multiplicity of $\V_{a,b}$ in $\V^{\otimes 2j}$, which is the number of walks of length $2j$ from $(0,0)$ to $(a,b)$ in the dominant chamber $a\ge b\ge0$ with steps $\pm e_1, \pm e_2$.

\begin{prop}\label{EF}
\[
\cF^{\bS}=\sum_{\substack{\ell(\lambda)\le4\\ |\lambda|\ \mathrm{even}}}\Big(\sum_{a\ge b\ge0}m^\lambda_{a,b}\,x^{(|\lambda|-a-b)/2}\,\HH_{\sA_2;\V_{a,b}}\Big)s_{\lambda'}\ -\ \frac12\Big((\cE^{\bS})^2+p_2\pl\cE^{\bS}\Big),
\]
with $\HH_{\sA_2;\V_{a,b}}$ given by \eqref{46}. Numerically, $\HH_{\cM_2}=x^3$ and for $j\ge1$,
\[
\HH_{\cM_2;\V^{\otimes 2j}}=\sum_{a\ge b\ge0}m^{2j}_{a,b}\,x^{(2j-a-b)/2}\,\HH_{\sA_2;\V_{a,b}}
\;-\;\tfrac12\sum_{i=0}^{j}\binom{2j}{2i} \HH_{\cM_{1,1};V^{\otimes 2i}} \HH_{\cM_{1,1};V^{\otimes (2j-2i)}},
\]
so that
\beq\label{56}
\cF(y)=x^3+\frac12x^2+\sum_{j> 0}\frac{y^{2j}}{(2j)!}\sum_{a\ge b\ge0}m^{2j}_{a,b}\,x^{(2j-a-b)/2}\,\HH_{\sA_2;\V_{a,b}}-\frac12\cE(y)^2.\eeq
\end{prop}
\begin{proof}
The Torelli morphism identifies $\cM_2$ with $\sA_2\smallsetminus\sA_{1,1}$ as stacks, where $\sA_{1,1}=[\sA_1^2/\bS_2]$ and $\sA_1=\cM_{1,1}$. Hence $\HH_{\cM_2;-}=\HH_{\sA_2;-}-\HH_{\sA_{1,1};-}$ for the local systems $\mathbb S_\lambda\V$, where $\V$ also denotes the standard local system on $\sA_2$.

For the first term, the representation $\mathbb S_\lambda(\CC^4)$ of $\mathrm{GL}_4$ has weight $|\lambda|$, so its restriction to $\mathrm{Sp}_4$ decomposes as $\bigoplus_{a,b}\V_{a,b}^{\oplus m^\lambda_{a,b}}$ with the summand $\V_{a,b}$, of weight $a+b$, occurring with Tate twist $(|\lambda|-a-b)/2$. This gives the first sum.

For the second, the restriction of $\V$ to $\sA_1^2$ is $V_1\oplus V_2$, the sum of the pullbacks of $V$ by the two projections. By Lemma \ref{fibered} (b) and the Cauchy identity, the $\bS$-object $r\mapsto H^*_c(\sA_1^2;(V_1[-1]\oplus V_2[-1])^{\otimes r})$ is the species square of the $\bS$-object $r\mapsto H^*_c(\sA_1;V[-1]^{\otimes r})$, by the K\"unneth formula, and the swap of $\sA_1^2$ exchanges the two factors. By \eqref{Burnside} and Lemma \ref{orbit} with $\ell=2$ and $\gamma=\id$, the contribution of $\sA_{1,1}$ is $\frac12((\cE^{\bS})^2+p_2\pl\cE^{\bS})$.

For the numerical form, apply the rank map. Since $\sum_{\lambda\vdash 2j}\dim V_{\lambda'}\mathbb S_\lambda=\otimes^{2j}$, the first term becomes the sum over $a,b$ with the multiplicities $m^{2j}_{a,b}$ of $\V_{a,b}$ in $\V^{\otimes 2j}$. The standard representation $\V$ has weights $\pm e_1,\pm e_2$, each of multiplicity one and none zero, so $\V_{a,b}\otimes\V=\bigoplus\V_{(a,b)\pm e_i}$, the sum being over those $(a,b)\pm e_i$ which are dominant. Iterating, $m^{2j}_{a,b}$ is the number of walks as stated. By Lemma \ref{rankmap}, $\rk(p_2\pl\cE^{\bS})=\psi^2(x)=x^2$, which together with $\HH_{\sA_2}=x^3+x^2$ gives $\HH_{\cM_2}=x^3$ and the term $\frac12x^2$ in \eqref{56}, and $\rk((\cE^{\bS})^2)=\cE(y)^2$ gives the last sum, where the terms with an odd number of factors $V_1$ vanish by Lemma \ref{35}.
\end{proof}

For instance, at $2j=2$, the walks give $m^2_{0,0}=m^2_{1,1}=m^2_{2,0}=1$ with twists $1,0,0$, so the first sum is $x(x^3+x^2)+(-1)+(-x)=x^4+x^3-x-1$, and the second is $\tfrac12\cdot2\cdot x(x^2-1)=x^3-x$. Hence $\HH_{\cM_2;\V^{\otimes 2}}=x^4-1$. Likewise, $\HH_{\cM_2;\V^{\otimes 4}}=3x^5-2x^4+3x^2-5x-3$. Together with \eqref{45} and \eqref{46}, Proposition \ref{EF} expresses $\cF^{\bS}$ in terms of the branching multiplicities and the dimensions $s_k$ and $s_{j,k}$ of spaces of cusp forms. All but the $s_{j,k}$ are elementary.

\subsection{Punctured universal curves}\label{S5.3}
In this subsection, we compute $A^{\bS,\beta}_{g,d}$ and $A^\beta_{g,d}$ in \eqref{defA} for those $(g,d)$ that appear in the computation of $\cL_2^\bS$ and $\cL_2$.
\begin{prop}\label{Agen0}
For $\beta\in\bS_d$ with $d\in\{3,4\}$, let $f_k(\beta)$ be the number of fixed points of $\beta^k$ on $[d]=\{1,2,\cdots,d\}$. Then
\[
A^{\bS,\beta}_{0,d}=\HH^\beta_{\cM_{0,d}}\cdot\exp\Big(\sum_{k\ge1}\big(x^k+1-f_k(\beta)\big)\frac{p_k}{k}\Big),
\]
with $\HH^\beta_{\cM_{0,3}}=1$, and $\HH^\beta_{\cM_{0,4}}=x-2,\ x,\ x+1$ according as the image of $\beta$ under the surjection $\bS_4\to\bS_3$, whose kernel is the Klein four group, is trivial, a transposition or a $3$-cycle. Explicitly, with $\bar\be=\be(-1)$,
\begin{align*}
A^{\bS}_{0,3}&=\bh(x)\bar\be^2, & A^{\bS,(12)}_{0,3}&=\bh(x)\,(p_2\pl\bar\be), & A^{\bS,(123)}_{0,3}&=\bh(x)\bh(1)\,(p_3\pl\bar\be),\\
A^{\bS}_{0,4}&=(x-2)\bh(x)\bar\be^3, & A^{\bS,(12)}_{0,4}&=x\,\bh(x)\bar\be\,(p_2\pl\bar\be),\\
A^{\bS,(12)(34)}_{0,4}&=(x-2)\bh(x)\bh(1)(p_2\pl\bar\be)^2, & A^{\bS,(1234)}_{0,4}&=x\,\bh(x)\bh(1)(p_4\pl\bar\be).
\end{align*}
\end{prop}
\begin{proof}
The fiber of $\mathring q_{0,d}$ is $\mathring C=\PP^1$ minus $d$ points. From $0\to H^0(\PP^1)\to H^0(\{d\text{ points}\})\to H^1_c(\mathring{C})\to 0$, we have $H^1_c(\mathring{C})=\QQ^d/\QQ$, of weight $0$, on which $\bS_d$ acts through the permutation representation of the $d$ punctures, while $H^2_c(\mathring C)=\QQ(-1)$ with trivial action. Since the punctures are labelled, the local systems $R^i\mathring q_{0,d!}\QQ$ are constant, so $K=R\mathring q_{0,d!}\QQ=\QQ(-1)[-2]\oplus(\QQ^d/\QQ)[-1]$ is a constant variation, $\beta$-equivariantly. The super trace of $\beta^k$ on $K$ is $x-(f_k(\beta)-1)$, and Lemma \ref{fibered} (a) with $K'=0$ gives the formula. 

The action of $\bS_3$ on $\cM_{0,3}=\mathrm{pt}$ is trivial. The action of $\bS_4$ on $\cM_{0,4}=\PP^1\smallsetminus\{0,1,\infty\}$ factors through $\bS_4\to\bS_3$, and $\bS_3$ permutes the three punctures $\{0,1,\infty\}$. By the computation above with $d=3$, the trace of a permutation of the three punctures on $H^*_c(\cM_{0,4})$ is $x$ minus the number of fixed punctures plus $1$, i.e.\ $x-2$, $x$ and $x+1$ for the identity, a transposition and a $3$-cycle. A transposition and a $4$-cycle of $\bS_4$ map to transpositions, a $3$-cycle maps to a $3$-cycle, and $(12)(34)$ maps to the identity.

For the explicit expressions, we use $\exp(-\sum_k p_k/k)=\bar\be$ and $\exp(-\sum_m p_{\ell m}/m)=p_\ell\pl\bar\be$. For instance, for $\beta=(123)$ we have $f_k=3$ if $3|k$ and $f_k=0$ otherwise, so the exponent is $\sum_k(x^k+1)p_k/k-3\sum_mp_{3m}/(3m)$, which gives $\bh(x)\bh(1)(p_3\pl\bar\be)$. The other cases are similar. 
\end{proof}

\begin{prop}\label{Agen1}
Let $\partial=\partial/\partial p_1$. Then
\begin{align*}
A^{\bS}_{2,0}&=\bh(1)\bh(x)\cF^{\bS}, \qquad A^{\bS}_{1,1}=\bh(x)\cE^{\bS},\\
A^{\bS}_{1,2}&=\bh(x)\be(-1)\big(x\cE^{\bS}+\partial\cE^{\bS}\big),\qquad
A^{\bS,(12)}_{1,2}=\bh(x)\be(1)\big(x\cE^{\bS}-\partial\cE^{\bS}\big).
\end{align*}
\end{prop}
\begin{proof}
For $A_{2,0}^\bS$, $K=R\mathring q_{2,0!}\QQ=\QQ\oplus\QQ(-1)[-2]\oplus\V[-1]$ by \eqref{19}, and Lemma \ref{fibered} gives $\bh(1)\bh(x)\cF^{\bS}$ as in Example \ref{39}. For $A_{1,1}^\bS$, $K=R\mathring q_{1,1!}\QQ=\QQ(-1)[-2]\oplus V[-1]$, and Lemma \ref{fibered} gives $\bh(x)\cE^{\bS}$.

For $A_{1,2}^\bS$ and $A^{\bS,(12)}_{1,2}$, let $B=\cM_{1,2}$ and $Y=\mathring\cC_{1,2}=\cM_{1,3}$. We write a point of $B$ as $(E,0,z)$ with $z\ne0$, so that the fiber of $Y\to B$ is $E\smallsetminus\{0,z\}$ and a point of $Y^{(r)}_B$ is $(E,0,z;p_1,\dots,p_r)$ with $p_i\notin\{0,z\}$. From $0\to H^0(E)\to H^0(\{0,z\})\to H^1_c(E\smallsetminus\{0,z\})\to H^1(E)\to0$, we see that $K=R\mathring q_{1,2!}\QQ$ has a filtration with graded pieces $\QQ(-1)[-2]$, $\varepsilon[-1]$ and $V[-1]$, where $\varepsilon=H^0(\{0,z\})/H^0(E)\cong\QQ$ is a constant variation of weight $0$ and $V$ is the pullback of $V$ from $\cM_{1,1}$. Since the class of $K$ in the Grothendieck group of equivariant variations is the sum of these graded pieces, Lemma \ref{fibered} (a) applies with $T=\QQ(-1)[-2]\oplus\varepsilon[-1]$ and $K'=V[-1]$.

Let $\iota=(12)$. Exchanging the two marked points of $(E,0,z)$ gives $(E,z,0)$, and translating by $-z$ to restore the origin gives $(E,0,-z)$. Hence $\iota$ acts on $Y^{(r)}_B$ by $(z;p_1,\dots,p_r)\mapsto(-z;p_1-z,\dots,p_r-z)$, covering the identity of $\cM_{1,1}$. On the fibers, $p\mapsto p-z$ maps $\{0,z\}$ to $\{-z,0\}$, exchanging the two punctures, so $\iota$ acts by $-1$ on $\varepsilon$ and acts trivially on $\QQ(-1)[-2]$. Since a translation acts trivially on $H^1(E)$, it acts on the $V$-part by the identity of $H^1(E)$, i.e.\ by the natural action on the pullback of $V$ from $\cM_{1,1}$. Thus $\str(\iota^k|T)=x-(-1)^k$ and $\str(\id|T)=x-1$, and the factor of Lemma~\ref{fibered}~(a) is $\bh(x)\be(1)$ for $\iota$ and $\bh(x)\be(-1)$ for the identity.

It remains to compute $\HH^\gamma_{B;\mathbb S_\lambda V}$ for $\gamma\in\{\id,\iota\}$. The morphism $\pi:B=\cM_{1,2}\to\cM_{1,1}$ has fibers $E\smallsetminus\{0\}$, and $\iota$ acts fiberwise by $z\mapsto-z$, which acts trivially on $H^2_c(E\smallsetminus\{0\})=\QQ(-1)$ and by $-1$ on $H^1_c(E\smallsetminus\{0\})=H^1(E)$. By the Leray spectral sequence for $\pi$ and the projection formula,
\[
\HH^\gamma_{B;\mathbb S_\lambda V}=\HH_{\cM_{1,1};\mathbb S_\lambda V\otimes(\QQ(-1)[-2]\pm V[-1])}=x\,\HH_{\cM_{1,1};\mathbb S_\lambda V}\mp\HH_{\cM_{1,1};\mathbb S_\lambda V\otimes V},
\]
with the upper sign for $\gamma=\id$ and the lower sign for $\gamma=\iota$, where for $\iota$ the eigenvalue $-1$ on $V[-1]$ changes its sign. By the Pieri rule, $\mathbb S_\lambda V\otimes V=\bigoplus_{\mu=\lambda+\square}\mathbb S_\mu V$, where $\mu$ runs over the partitions obtained from $\lambda$ by adding a box. Since $\partial s_{\mu'}=\sum_{\lambda=\mu-\square}s_{\lambda'}$,
\[
\sum_\lambda(-1)^{|\lambda|}s_{\lambda'}\HH_{\cM_{1,1};\mathbb S_\lambda V\otimes V}=\sum_\mu(-1)^{|\mu|-1}\HH_{\cM_{1,1};\mathbb S_\mu V}\,\partial s_{\mu'}=-\partial\cE^{\bS}.
\]
Hence $\sum_\lambda(-1)^{|\lambda|}s_{\lambda'}\HH^\gamma_{B;\mathbb S_\lambda V}=x\cE^{\bS}\pm\partial\cE^{\bS}$, which, multiplied by the factor computed above, gives $A^\bS_{1,2}$ and $A^{\bS,(12)}_{1,2}$.
\end{proof}


Applying the rank map, with $\rk(\partial f)=\frac{d}{dy}\rk(f)$ and Lemma \ref{rankmap}, we obtain the following.
\begin{coro}\label{36}
$A_{2,0}=e^{(1+x)y}\mathcal F(y)$, $A_{1,1}=e^{xy}\mathcal E(y)$, $A_{0,3}=e^{(x-2)y}$, $A_{0,4}=(x-2)e^{(x-3)y}$,
\[
A_{1,2}=e^{(x-1)y}\big(x\mathcal E(y)+\mathcal E'(y)\big),\ \
A^{(1,2)}_{1,2}(y)=e^{(1+x)y}\big(x\mathcal E(y)-\mathcal E'(y)\big),
\]
\[A_{0,3}^{(1,2)}= e^{xy},\ \ A_{0,3}^{(1,2,3)}=e^{(x+1)y},\ \ A_{0,4}^{(1,2)}=xe^{(x-1)y},\]
\[A_{0,4}^{(1,2)(3,4)}=(x-2)e^{(x+1)y} ,\ \ A_{0,4}^{(1,2,3,4)}=xe^{(x+1)y}. \]
\end{coro}

\bigskip
\section{The genus two formula}\label{S6}
In this section, we assemble the results of \S\ref{S4} and \S\ref{S5} into a closed formula for $\cL^{\bS}_2$ and prove Theorem C.

By Theorem \ref{graphsum} and Proposition \ref{28},
\beq\label{47}\cL^{\bS}_2=\sum_{n\ge 0}\HH^{\bS}_{\bcM_{2,(0^+)^n}}=\cL^{\bS}_\varnothing +\cL^{\bS}_{\Gamma_1}+\cL^{\bS}_{\Gamma_0}+\cL^{\bS}_{\Gamma_a}+\cL^{\bS}_{\Gamma_b}+\cL^{\bS}_{\Gamma_c}+\cL^{\bS}_{\Gamma_d}.\eeq

\def\bC{\mathbf{C}}
\def\bD{\mathbf{D}}
\def\bA{\mathbf{A}}
\begin{lemm}\label{54}
Write $\bA=A^\bS$, $\bC=C^{\bS}$ and $\bD=D^{\bS}$. Then
\begin{align*}
\cL^{\bS}_{\varnothing}&=\bA_{2,0},\\
\cL^{\bS}_{\Gamma_1}&=\tfrac12\Big[\bA_{1,1}^2\bC+(p_2\pl \bA_{1,1})\bD\Big],\\
\cL^{\bS}_{\Gamma_0}&=\tfrac12\Big[\bA_{1,2}\bC +\bA_{1,2}^{(12)}\bD\Big],\\
\cL^{\bS}_{\Gamma_a}&=\tfrac12\Big[\bA_{0,3}\bA_{1,1}\bC^2+\bA_{0,3}^{(12)}\bA_{1,1}\bC\bD\Big],\\
\cL^{\bS}_{\Gamma_b}&=\tfrac18\Big[ \bA_{0,4}\bC^2+2\bA_{0,4}^{(12)}\bC\bD+\bA_{0,4}^{(12)(34)}\bD^2+2\bA_{0,4}^{(12)(34)}(p_2\pl \bC)+2\bA_{0,4}^{(1234)}(p_2\pl \bD)\Big],\\
\cL^{\bS}_{\Gamma_c}&=\tfrac1{12}\Big[ \bA_{0,3}^2\bC^3+3\big(\bA_{0,3}^{(12)}\big)^2\bC\,(p_2\pl \bC)+2\big(\bA_{0,3}^{(123)}\big)^2(p_3\pl \bC)\\
&\qquad\quad+(p_2\pl \bA_{0,3})\big(\bD^3+3\bD\,(p_2\pl \bC)\big)+2\big(p_2\pl \bA_{0,3}^{(123)}\big)(p_3\pl \bD)\Big],\\
\cL^{\bS}_{\Gamma_d}&=\tfrac18\Big[\bA_{0,3}^2\bC^3+2\bA_{0,3}^{(12)}\bA_{0,3}\bC^2\bD+\big(\bA_{0,3}^{(12)}\big)^2\bC\bD^2\\
&\qquad\quad+2(p_2\pl \bA_{0,3})(p_2\pl \bC)\bD+2\big(p_2\pl \bA_{0,3}^{(12)}\big)(p_2\pl \bD)\bD\Big].
\end{align*}
\end{lemm}
\begin{proof} This is immediate from \eqref{23S}.
\end{proof}

\def\bE{\mathbf{E}}
\def\bF{\mathbf{F}}
We thus obtain the following genus two formula. 
\begin{theo}\label{L2S}
Write $\bE=\cE^{\bS}$, $\bF=\cF^{\bS}$, $\bC=C^{\bS}$, $\bD=D^{\bS}$ and $\bar\be=\be(-1)$. Then $\cL^{\bS}_2$ is the sum of the following seven terms:
\begin{align*}
\cL^{\bS}_{\varnothing}&=\bh(1)\bh(x)\bF,\\
\cL^{\bS}_{\Gamma_1}&=\tfrac12\Big[\bh(x)^2\bE^2\bC+\big(p_2\pl(\bh(x)\bE)\big)\bD\Big],\\
\cL^{\bS}_{\Gamma_0}&=\tfrac12\Big[\bh(x)\bar\be\,(x\bE+\partial\bE)\bC+\bh(x)\be(1)\,(x\bE-\partial\bE)\bD\Big],\\
\cL^{\bS}_{\Gamma_a}&=\tfrac12\Big[\bh(x)^2\bar\be^2\bE \bC^2+\bh(x)^2(p_2\pl\bar\be)\bE \bC\bD\Big],\\
\cL^{\bS}_{\Gamma_b}&=\tfrac18\Big[(x-2)\bh(x)\bar\be^3\bC^2+2x\,\bh(x)\bar\be\,(p_2\pl\bar\be)\bC\bD\\
&\qquad+(x-2)\bh(x)\bh(1)(p_2\pl\bar\be)^2\big(\bD^2+2\,p_2\pl \bC\big)+2x\,\bh(x)\bh(1)(p_4\pl\bar\be)(p_2\pl \bD)\Big],\\
\cL^{\bS}_{\Gamma_c}&=\tfrac1{12}\Big[\bh(x)^2\bar\be^4\bC^3+3\,\bh(x)^2(p_2\pl\bar\be)^2\bC\,(p_2\pl \bC)+\big(p_2\pl(\bh(x)\bar\be^2)\big)\big(\bD^3+3\bD\,(p_2\pl \bC)\big)\\
&\qquad+2\,\bh(x)^2\bh(1)^2(p_3\pl\bar\be)^2(p_3\pl \bC)+2\big(p_2\pl(\bh(x)\bh(1)(p_3\pl\bar\be))\big)(p_3\pl \bD)\Big],\\
\cL^{\bS}_{\Gamma_d}&=\tfrac18\Big[\bh(x)^2\bar\be^4\bC^3+2\,\bh(x)^2\bar\be^2(p_2\pl\bar\be)\bC^2\bD+\bh(x)^2(p_2\pl\bar\be)^2\bC\bD^2\\
&\qquad+2\big(p_2\pl(\bh(x)\bar\be^2)\big)(p_2\pl \bC)\bD+2\big(p_2\pl(\bh(x)(p_2\pl\bar\be))\big)(p_2\pl \bD)\bD\Big],
\end{align*}
where $\bC=C^\bS$ and $\bD=D^\bS$ are given by Proposition \ref{LM}, $\bE=\cE^\bS$ by Proposition \ref{40} and $\bF=\cF^\bS$ by Proposition \ref{EF}.
\end{theo}
\begin{proof}
Substitute Propositions \ref{Agen0} and \ref{Agen1} into Lemma \ref{54}. 
\end{proof}

\begin{coro}[Theorem C]\label{37}
With $C(y), D(y)$ as in Corollary~\ref{25} and $\mathcal E(y),\mathcal F(y)$ as in \S\ref{S5.2}, $\cL_2(y)$ is the sum of the following seven terms:
\begin{align*}
\cL_{\varnothing}&=e^{(1+x)y}\mathcal F(y),\\
\cL_{\Gamma_1}&=\tfrac12\Big[e^{2xy}\mathcal E^2C+x^2D\Big],\\
\cL_{\Gamma_0}&=\tfrac12\Big[e^{(x-1)y}\big(x\mathcal E+\mathcal E'\big)C
              +e^{(1+x)y}\big(x\mathcal E-\mathcal E'\big)D\Big],\\
\cL_{\Gamma_a}&=\tfrac12\Big[e^{(2x-2)y}\mathcal EC^2+e^{2xy}\mathcal ECD\Big],\\
\cL_{\Gamma_b}&=\tfrac18\Big[(x-2)e^{(x-3)y}C^2+2x\,e^{(x-1)y}CD
              +(x-2)e^{(x+1)y}D^2+(4x-4)e^{(x+1)y}\Big],\\
\cL_{\Gamma_c}&=\tfrac1{12}\Big[e^{2(x-2)y}C^3+3e^{2xy}C+2e^{2(x+1)y}
              +D^3+3D+2\Big],\\
\cL_{\Gamma_d}&=\tfrac18\Big[e^{2(x-2)y}C^3+2e^{(2x-2)y}C^2D
              +e^{2xy}CD^2+4D\Big].
\end{align*}
\end{coro}
\begin{proof}
Apply the rank map to Theorem \ref{L2S}. By Lemma \ref{rankmap}, $\rk(\bh(t))=\rk(\be(t))=e^{ty}$, $\rk(\partial f)=\frac{d}{dy}\rk(f)$, and a factor $p_\ell\pl f$ with $\ell\ge2$ contributes the constant term $\psi^\ell(f(0))$: this is $\psi^2(x)=x^2$ for $f=\bh(x)\cE$ and $1$ for $f=\bar\be$, $C$, $D$, $\bh(x)\bar\be^2$, $\bh(x)(p_2\pl\bar\be)$ and $\bh(x)\bh(1)(p_3\pl\bar\be)$. In $\cL_{\Gamma_b}$, the terms $2(x-2)e^{(x+1)y}+2xe^{(x+1)y}$ combine to $(4x-4)e^{(x+1)y}$.
\end{proof}

All the terms in Theorem \ref{L2S} are readily computable except $\cF^{\bS}$, which requires the dimensions $s_{a-b,b+3}$ and $s_{0,b+3}$ for $a\ge b\ge 0$ with $a+b$ even. We thus obtain the following.
\begin{coro}\label{62}
Theorem \ref{L2S} computes $\HH^{\bS}_{\bcM_{2,(0^+)^n}}$, and hence $\HH^{\bS}_{\bcM_{2,n}}$ by Theorem \ref{A}, for every $n$ such that $s_{j,k}=\dim S_{j,k}(\mathrm{Sp}_4(\Z))$ is known for all even $j\ge0$ and $k\ge3$ with $j+2k\le 2\lfloor n/2\rfloor+6$. For $n\le19$, the required dimensions are exactly
\[
s_{0,10}=s_{0,12}=s_{6,8}=s_{8,8}=s_{4,10}=s_{12,6}=1,\qquad\text{all others }0,
\]
so the computation is unconditional for all $n\le 19$. Since the dimensions $s_{j,k}$ contribute only to odd degrees, the $\bS_n$-representations on all $H^{p,q}(\bcM_{2,(0^+)^n})$ and $H^{p,q}(\bcM_{2,n})$ with $p+q$ even are unconditionally determined by Theorem \ref{L2S} for every $n$.
\end{coro}
\begin{proof}
The part of $\cF^{\bS}$ of degree at most $n$ involves the local systems $\mathbb S_\lambda\V$ with $|\lambda|\le n$ even, hence $\V_{a,b}$ with $a+b\le 2\lfloor n/2\rfloor$, and by \eqref{46} the dimensions $s_{j,k}$ with $j=a-b$ and $k=b+3$, so $j+2k=a+b+6$. The listed values are classical (cf.\ \cite{Pet}). The dimensions $s_{j,k}$ enter \eqref{46} only through monomials of odd total degree, since $j$ is even. They enter $\cL^{\bS}_2$ linearly, through $\cL^{\bS}_\varnothing=\bh(1)\bh(x)\cF^{\bS}$, multiplied by series whose coefficients are polynomials in $x$, and the plethysm $\cL^{\bS}_2\pl Z^{\langle-1\rangle}$ acts on the coefficients of $\cL^{\bS}_2$ linearly, multiplying them by polynomials in $x$. Hence they affect only monomials $u^pv^q$ with $p+q$ odd.
\end{proof}

\bigskip
\section{Computations and applications}\label{S7}

\subsection{Implementation and checks}\label{S7.1}
We implemented Theorems \ref{A}, \ref{graphsum} and \ref{L2S} in exact rational arithmetic, representing elements of $\Lam$ in the power sum basis with coefficients in $\QQ[u,v]$, truncated in degree, and converting to Schur functions by the Murnaghan-Nakayama rule. The plethystic inverse $Z^{\langle-1\rangle}$ is obtained from the fixed point iteration of Remark \ref{Arem}. With this exact implementation, we computed $\HH^{\bS}_{\bcM_{2,n}}$ in all degrees for $n\le18$, and 
$\HH^{\bS}_{\bcM_{2,n}}$ coincides with the $\bS_n$-equivariant Euler characteristic of \cite{Ber}, obtained independently by counting points over finite fields, Schur function by Schur function and Hodge type by Hodge type, under $L\mapsto x$, $\mathsf S_k\mapsto u^{k-1}+v^{k-1}$ and $\mathsf S_{j,k}\mapsto W_{j,k}(u,v)$. This includes all terms involving $\mathsf S_{12},\mathsf S_{16},\mathsf S_{18},\mathsf S_{20},\mathsf S_{22}$ and the Siegel motives $\mathsf S_{6,8},\mathsf S_{8,8},\mathsf S_{4,10}$. 

\subsection{Betti numbers}\label{S7.2}
The following is the list of the Betti numbers $b_i(\bcM_{2,n})$ for $1\le n\le 19$ computed by Corollary \ref{37} and Proposition \ref{7}. They are listed as $b_0, b_2, b_4, \cdots$ up to the middle, the remaining being determined by Poincar\'e duality; nonzero odd Betti numbers are listed separately.
\begin{footnotesize}\raggedright\setlength{\emergencystretch}{6em}
\begin{itemize}[leftmargin=2.6em,itemsep=1.5pt,rightmargin=-1.2em]
\item[$n=1$:] 1, 3, 5
\item[$n=2$:] 1, 6, 14
\item[$n=3$:] 1, 12, 44, 67
\item[$n=4$:] 1, 24, 144, 333
\item[$n=5$:] 1, 48, 474, 1668, 2501
\item[$n=6$:] 1, 96, 1547, 8256, 18296
\item[$n=7$:] 1, 192, 4986, 39969, 129342, 189289
\item[$n=8$:] 1, 384, 15870, 188697, 880649, 1849853
\item[$n=9$:] 1, 768, 49950, 869522, 5774876, 17127132, 24427499
\item[$n=10$:] 1, 1536, 155713, 3920229, 36556626, 150785735, 300496807;\; odd: $b_{13}{=}2$
\item[$n=11$:] 1, 3072, 481534, 17340292, 224167059, 1267763370, 3469316099, 4828313319;\; odd: $b_{13}{=}44$
\item[$n=12$:] 1, 6144, 1479324, 75452435, 1336540026, 10227205166, 37841246633, 71850818691;\; odd: $b_{13}{=}528$, $b_{15}{=}1672$
\item[$n=13$:] 1, 12288, 4520434, 323751896, 7775916100, 79534380552, 392351726031, 999268944378, 1359879381775;\; odd: $b_{13}{=}4576$, $b_{15}{=}33748$
\item[$n=14$:] 1, 24576, 13754503, 1372779858, 44288590812, 598892734287, 3889207335069, 13089118559365, 23794703831110;\; odd: $b_{13}{=}32032$, $b_{15}{=}468468$, $b_{17}{=}1073670$
\item[$n=15$:] 1, 49152, 41710834, 5762965223, 247657578541, 4384047893932, 37050861383436, 162586528339655, 388563330869810, 518215971346233;\; odd: $b_{13}{=}192192$, $b_{15}{=}5065060$, $b_{17}{=}22665942$
\item[$n=16$:] 1, 98304, 126157562, 23990725337, 1363084174826, 31309234326027, 340825440518296, 1926857354714685, 5969233095899064, 10437928997918556;\; odd: $b_{13}{=}1025024$, $b_{15}{=}45804616$, $b_{17}{=}361450596$, $b_{19}{=}707075048$
\item[$n=17$:] 1, 196608, 380800502, 99171904306, 7400382446483, 218821428425418, 3039999109374664, 21905585131953561, 86870342196086929, 196238344880596871, 257023830913166395;\; odd: $b_{13}{=}4978688$, $b_{15}{=}362189828$, $b_{17}{=}4705416196$, $b_{19}{=}16185231684$
\item[$n=18$:] 1, 393216, 1147645933, 407557387095, 39707156779317, 1500738117616881, 26388426769441311, 240035299025783896, 1204960112715584960, 3470729810061890755, 5862207384250863896;\; odd: $b_{13}{=}22404096$, $b_{15}{=}2579397964$, $b_{17}{=}52567385312$, $b_{19}{=}294944677504$, $b_{21}{=}520847892134$
\item[$n=19$:] 1, 786432, 3454670582, 1666760113266, 210900933524186, 10123795363087172, 223636953444606116, 2545851833702518589, 16015974374873130211, 58140062308346077328, 124928421013039794593, 160994260381274623229;\; odd: $b_{13}{=}94595072$, $b_{15}{=}16892795348$, $b_{17}{=}521282033272$, $b_{19}{=}4520519253012$, $b_{21}{=}13046444012836$
\end{itemize}\end{footnotesize}
For $n\le 18$, these numbers agree with \cite{Ber}, which was obtained by an independent method.

\medskip
The even degree Betti numbers can be computed for any $n$ by Corollary \ref{62}, within the limits of computing resources. The following is the list of even degree Betti numbers $b_{2i}$ of $\bcM_{2,n}$ for $20\le n\le 30$.
\begin{scriptsize}\raggedright\setlength{\emergencystretch}{6em}
\begin{itemize}[leftmargin=2.6em,itemsep=2pt,rightmargin=-1.2em]
\item[$n=20$:] 1, 1572864, 10390097240, 6788967928547, 1110449368995637, 67312967783390146, 1855577897064966799, 26231656835173147385, 204956138295013257688, 927957831395183416149, 2506207708425389496658, 4103924244941702378515;\; then symmetric by duality.
\item[$n=21$:] 1, 3145728, 31227965930, 27560365896424, 5803115546177229, 441926519039333550, 15110637767035766181, 263421107129957808644, 2535659787403291098951, 14185426510619816828634, 47637435889967658937557, 97894835427218888126277, 124334448501272731093443;\; then symmetric by duality.
\item[$n=22$:] 1, 6291456, 93810778419, 111576179375916, 30131782967467137, 2869308074796375501, 121027583326470029208, 2585424423744717116975, 30438124401232289237290, 208643551740682373074748, 862827790064231388955992, 2200601363120508763222088, 3504739020781363336182808;\; then symmetric by duality.
\item[$n=23$:] 1, 12582912, 281710211274, 450687852851565, 155591399163935606, 18449051850314668338, 955211410306486848898, 24862430810995162354161, 355650193065102381640541, 2964566621592039188423811, 14966256969682129731720883, 46903693409464498001988913, 92611273193924710760621808, 116094260133573453408044295;\; then symmetric by duality.
\item[$n=24$:] 1, 25165824, 845736714614, 1817083705496937, 799622406796722839, 117614304338982437915, 7441234437817057755645, 234763311986523109572817, 4056211193209523379555502, 40836213491596315476152274, 249703200569530212928325024, 953007253877444650080299131, 2309201985769333344660304648, 3586864947030781583136217552;\; then symmetric by duality.
\item[$n=25$:] 1, 50331648, 2538527159694, 7315018450336578, 4092769070505916422, 744200933804508018532, 57299653723595404845391, 2180842617818148407435840, 45266808440862592407849121, 547026854347634670507313257, 4022892123778066183188518667, 18546877858783783846566414353, 54644679811512926074824735915, 104085196163354920816614657760, 128944613163974356780618522239;\; then symmetric by duality.
\item[$n=26$:] 1, 100663296, 7618433610777, 29411600285657097, 20875455171152188148, 4678026610548637741259, 436694167851371281497525, 19964649411564915543575451, 495388038587975416991671760, 7145665431743738257685673830, 62797120822371748924365618201, 347177688641887505531945094716, 1233462205034226607309041504263, 2854159414431862335676042091906, 4334295662729170193244423076311;\; then symmetric by duality.
\item[$n=27$:] 1, 201326592, 22861458076206, 118136656961188256, 106159616118066133517, 29236604493578955525126, 3297716504321726239806969, 180382427750215154364113508, 5326647644845893944840461112, 91243471642131513898153958459, 952663873753142204618445002213, 6274077445467378217747761521361, 26677586319137994070980793620483, 74356276814603904079260326415152, 137102690953460737406601908865947, 168038815025526354165044911236453;\; then symmetric by duality.
\item[$n=28$:] 1, 402653184, 68597628235540, 474128183174425635, 538486928953790524476, 181798152179301213740894, 24699912664405216433690201, 1610648878795083648146322280, 56370009443749098483911529483, 1141355745881043521628520321290, 14083163583135238781264121831110, 109822692881429553058949110278763, 555059988456851567388892209826760, 1849170923836513683305957637684080, 4104317343031403263798808226163790, 6106347997110700598378020324739724;\; then symmetric by duality.
\item[$n=29$:] 1, 805306368, 205821338871970, 1901601066337595656, 2725491733677375970626, 1125432276637520150809212, 183658006666108574204562087, 14229736541719480994100524892, 588016401168247388693614572108, 14013027849795204031489136327161, 203354271799664304810342762230110, 1867434037741972601280293420898502, 11149176714551742419729740231123904, 44085585606514575854197582296539463, 116864315561984369199652789262846785, 209219705273510380800603371220578363, 253943624426244256706970996467486347;\; then symmetric by duality.
\item[$n=30$:] 1, 1610612736, 617524951472223, 7622734254808073894, 13769150396608759042098, 6939896011815182735396915, 1356748741367446789544633163, 124520722150222113097135504209, 6054363943466834154765520480986, 169151722185036820099033178205047, 2874206430634335204182005776794328, 30926904748770490522846713617510025, 216881578337956966497716323094973427, 1011382453526321054963483885734704726, 3179138629942668279786864412324616294, 6794394643209073297938553715734097797, 9921781367043583627366829407357601433;\; then symmetric by duality.
\end{itemize}\end{scriptsize}

\subsection{$\bS_n$-representations for small $n$}\label{S7.3}
For $n\le 9$, all the cohomology of $\bcM_{2,n}$ is of Tate type. The following table lists the $\bS_n$-representations $H^{2k}(\bcM_{2,n})=H^{k,k}(\bcM_{2,n})$ for $n\le7$ and $2k\le n+3$, as sums $\sum_\lambda m_\lambda s_\lambda$ meaning $\bigoplus_\lambda V_\lambda^{\oplus m_\lambda}$; the remaining degrees follow by Poincar\'e duality, which is $\bS_n$-equivariant. The data for all degrees and $n\le18$, and for even degrees and $n\le30$, are available from the authors.

{\small
\begin{longtable}{@{}l p{0.86\textwidth}@{}}
\toprule
\multicolumn{2}{l}{$n=1$}\\
$H^{0}$ & $s_{1}$\\
$H^{2}$ & $3s_{1}$\\
$H^{4}$ & $5s_{1}$\\
\midrule
\multicolumn{2}{l}{$n=2$}\\
$H^{0}$ & $s_{2}$\\
$H^{2}$ & $5s_{2}+s_{1^{2}}$\\
$H^{4}$ & $11s_{2}+3s_{1^{2}}$\\
\midrule
\multicolumn{2}{l}{$n=3$}\\
$H^{0}$ & $s_{3}$\\
$H^{2}$ & $6s_{3}+3s_{21}$\\
$H^{4}$ & $20s_{3}+12s_{21}$\\
$H^{6}$ & $28s_{3}+19s_{21}+s_{1^{3}}$\\
\midrule
\multicolumn{2}{l}{$n=4$}\\
$H^{0}$ & $s_{4}$\\
$H^{2}$ & $8s_{4}+4s_{31}+2s_{2^{2}}$\\
$H^{4}$ & $33s_{4}+26s_{31}+12s_{2^{2}}+3s_{21^{2}}$\\
$H^{6}$ & $67s_{4}+60s_{31}+28s_{2^{2}}+10s_{21^{2}}$\\
\midrule
\multicolumn{2}{l}{$n=5$}\\
$H^{0}$ & $s_{5}$\\
$H^{2}$ & $9s_{5}+6s_{41}+3s_{32}$\\
$H^{4}$ & $49s_{5}+48s_{41}+31s_{32}+8s_{31^{2}}+6s_{2^{2}1}$\\
$H^{6}$ & $128s_{5}+156s_{41}+106s_{32}+42s_{31^{2}}+26s_{2^{2}1}+s_{21^{3}}$\\
$H^{8}$ & $181s_{5}+227s_{41}+159s_{32}+65s_{31^{2}}+43s_{2^{2}1}+3s_{21^{3}}$\\
\midrule
\multicolumn{2}{l}{$n=6$}\\
$H^{0}$ & $s_{6}$\\
$H^{2}$ & $11s_{6}+7s_{51}+5s_{42}+s_{3^{2}}$\\
$H^{4}$ & $68s_{6}+75s_{51}+62s_{42}+17s_{41^{2}}+21s_{3^{2}}+16s_{321}+3s_{2^{3}}$\\
$H^{6}$ & $229s_{6}+317s_{51}+292s_{42}+118s_{41^{2}}+108s_{3^{2}}+115s_{321}+9s_{31^{3}}+22s_{2^{3}}+6s_{2^{2}1^{2}}$\\
$H^{8}$ & $420s_{6}+641s_{51}+615s_{42}+278s_{41^{2}}+236s_{3^{2}}+277s_{321}+29s_{31^{3}}+53s_{2^{3}}+21s_{2^{2}1^{2}}$\\
\midrule
\multicolumn{2}{l}{$n=7$}\\
$H^{0}$ & $s_{7}$\\
$H^{2}$ & $12s_{7}+9s_{61}+6s_{52}+3s_{43}$\\
$H^{4}$ & $90s_{7}+109s_{61}+100s_{52}+28s_{51^{2}}+58s_{43}+34s_{421}+12s_{3^{2}1}+8s_{32^{2}}$\\
$H^{6}$ & $363s_{7}+580s_{61}+606s_{52}+258s_{51^{2}}+389s_{43}+331s_{421}+28s_{41^{3}}+140s_{3^{2}1}+91s_{32^{2}}+34s_{321^{2}}+10s_{2^{3}1}$\\
$H^{8}$ & $854s_{7}+1529s_{61}+1728s_{52}+831s_{51^{2}}+1153s_{43}+1133s_{421}+143s_{41^{3}}+489s_{3^{2}1}+335s_{32^{2}}+170s_{321^{2}}+4s_{31^{4}}+47s_{2^{3}1}+2s_{2^{2}1^{3}}$\\
$H^{10}$ & $1125s_{7}+2109s_{61}+2430s_{52}+1221s_{51^{2}}+1647s_{43}+1675s_{421}+228s_{41^{3}}+738s_{3^{2}1}+502s_{32^{2}}+275s_{321^{2}}+7s_{31^{4}}+77s_{2^{3}1}+6s_{2^{2}1^{3}}$\\
\midrule
\end{longtable}}

For example, $H^6(\bcM_{2,3})$ contains the alternating representation once, while $H^2(\bcM_{2,n})$ contains only representations $V_\lambda$ with at most two rows, by \S\ref{S7.1} (iv).

\subsection{Odd degree cohomology}\label{S7.4}
For $n\le 14$, the odd degree cohomology is as follows. We list the Hodge type $(p,q)$ with $p>q$, the corresponding motive, as identified by the Hodge types and Theorem \ref{42}, and the $\bS_n$-representation on $H^{p,q}$. The representation on $H^{q,p}$ is the same, and the remaining degrees follow by Poincar\'e duality. For $n\le 9$, there is no odd degree cohomology.

{\small
\begin{longtable}{@{}l l l l p{0.62\textwidth}@{}}
\toprule
$n$ & & $(p,q)$ & motive & $H^{p,q}(\bcM_{2,n})$\\
\midrule
$10$ & $H^{13}$ & $(12,1)$ & $\mathsf S_{12}(-1)$ & $s_{1^{10}}$\\
\midrule
$11$ & $H^{13}$ & $(12,1)$ & $\mathsf S_{12}(-1)$ & $2s_{21^{9}}+2s_{1^{11}}$\\
\midrule
$12$ & $H^{13}$ & $(12,1)$ & $\mathsf S_{12}(-1)$ & $3s_{31^{9}}+s_{2^{2}1^{8}}+4s_{21^{10}}+s_{1^{12}}$\\
$12$ & $H^{15}$ & $(13,2)$ & $\mathsf S_{12}(-2)$ & $s_{41^{8}}+7s_{31^{9}}+3s_{2^{2}1^{8}}+11s_{21^{10}}+3s_{1^{12}}$\\
\midrule
$13$ & $H^{13}$ & $(12,1)$ & $\mathsf S_{12}(-1)$ & $4s_{41^{9}}+2s_{321^{8}}+6s_{31^{10}}+2s_{2^{2}1^{9}}+2s_{21^{11}}$\\
$13$ & $H^{15}$ & $(13,2)$ & $\mathsf S_{12}(-2)$ & $3s_{51^{8}}+2s_{421^{7}}+19s_{41^{9}}+12s_{321^{8}}+29s_{31^{10}}+s_{2^{3}1^{7}}+14s_{2^{2}1^{9}}+14s_{21^{11}}+s_{1^{13}}$\\
\midrule
$14$ & $H^{13}$ & $(12,1)$ & $\mathsf S_{12}(-1)$ & $5s_{51^{9}}+3s_{421^{8}}+8s_{41^{10}}+s_{3^{2}1^{8}}+4s_{321^{9}}+3s_{31^{11}}+s_{2^{2}1^{10}}$\\
$14$ & $H^{15}$ & $(13,2)$ & $\mathsf S_{12}(-2)$ & $6s_{61^{8}}+6s_{521^{7}}+37s_{51^{9}}+3s_{431^{7}}+s_{42^{2}1^{6}}+34s_{421^{8}}+60s_{41^{10}}+12s_{3^{2}1^{8}}+5s_{32^{2}1^{7}}+46s_{321^{9}}+34s_{31^{11}}+5s_{2^{3}1^{8}}+18s_{2^{2}1^{10}}+5s_{21^{12}}$\\
$14$ & $H^{17}$ & $(14,3)$ & $\mathsf S_{12}(-3)$ & $s_{621^{6}}+14s_{61^{8}}+15s_{521^{7}}+72s_{51^{9}}+s_{4^{2}1^{6}}+9s_{431^{7}}+3s_{42^{2}1^{6}}+70s_{421^{8}}+113s_{41^{10}}+26s_{3^{2}1^{8}}+12s_{32^{2}1^{7}}+94s_{321^{9}}+69s_{31^{11}}+13s_{2^{3}1^{8}}+40s_{2^{2}1^{10}}+13s_{21^{12}}$\\
$14$ & $H^{17}$ & $(16,1)$ & $\mathsf S_{16}(-1)$ & $2s_{1^{14}}$\\
$14$ & $H^{17}$ & $(17,0)$ & $\mathsf S_{18}$ & $s_{2^{7}}$\\
\midrule
\end{longtable}}

\begin{proof}[Proof of Theorem E]
This is read off from the table, which is part of the computation of \S\ref{S7.1}. Note that $\dim V_{2^7}=429$ is the Catalan number $C_7$. The classes of type $(17,0)$ in $H^{17}(\bcM_{2,14})$ come from the term $s_{18}(u^{17}+v^{17})$ of $\HH_{\Sbar_{0,10}}$ in Theorem \ref{42}, i.e.\ from the Saito-Kurokawa lift $\chi_{10}$ of the cusp form of weight $18$.
\end{proof}


\subsection{Non-Tate classes in even degree}\label{S7.5}
Theorem~\ref{L2S} is a statement about Hodge structures, not merely Betti numbers, and the extra information is not vacuous.

\begin{theo}\label{59}
For $n\le20$, every even degree Hodge number of $\Mbar_{2,n}$ lies on the diagonal, i.e. $$h^{p,q}(\bcM_{2,n})=0\ \  \text{ if }\ p\ne q,\  p+q \text{ even},\ n\le 20.$$
For $n=21$, this fails: the only off-diagonal even degree Hodge pieces of $\bcM_{2,21}$ are $H^{23,1}$ and $H^{1,23}$, and as $\bS_{21}$-representations
\[
\ch_{21}H^{23,1}(\bcM_{2,21})=\ch_{21}H^{1,23}(\bcM_{2,21})=s_1\cdot h_2[e_{10}]=\sum_{k=0}^{10}s_{2^k1^{21-2k}}+\sum_{a\in\{1,3,5,7,9\}}s_{3\,2^a1^{18-2a}}.
\]
In particular $h^{23,1}(\bcM_{2,21})=h^{1,23}(\bcM_{2,21})=1939938=21\cdot\tfrac12\tbinom{20}{10}$, carried by $\Sym^2\mathsf S_{12}(-1)$.
\end{theo}
\begin{proof}
By Corollary \ref{62}, all even degree pieces of $\HH^{\bS}_{\bcM_{2,n}}$ are unconditionally determined by Theorem \ref{L2S} and Theorem \ref{A}, and the first two statements follow by computing them for $n\le 21$, which we did both with the exact implementation (for $n\le 18$) and with the modular one.

The structure of the result can be seen as follows. By \eqref{45} and \eqref{46}, all monomials $u^pv^q$ with $p\ne q$ in $\cE^{\bS}$ and in the $\sA_2$-part of $\cF^{\bS}$ have $p+q$ odd. All terms of Theorem \ref{L2S} are linear in these series, with coefficients which are polynomials in $x$, except for the terms which are quadratic in $\cE^{\bS}$, namely $\cE^2$ and $p_2\pl(\bh(x)\cE)$ in $\cL^{\bS}_{\Gamma_1}$ and $(\cE^{\bS})^2$ and $p_2\pl\cE^{\bS}$ in $\cF^{\bS}$; and the plethysm with $Z^{\langle-1\rangle}$ preserves this property. Hence a monomial $u^pv^q$ with $p\ne q$ and $p+q$ even can only arise from these quadratic terms, in which two classes of type $(k-1,0)$ or $(0,k-1)$ coming from $\mathsf S_k$ are multiplied, or such a class is raised to the second Adams power. For $n\le 21$, only $k=12$ contributes, and the classes of type $(22,0)$ in $\mathsf S_{12}\otimes\mathsf S_{12}$ lie in $\Sym^2\mathsf S_{12}$, since $H^{11,0}$ of $\mathsf S_{12}$ is one-dimensional.

To isolate $u^{23}v$, we may also compute in the quotient $\QQ[u,v]/(u^{24},v^2)$, which is compatible with products and with the Adams operations. In this quotient $x^2=0$ and $p_k\pl x=0$ for $k\ge2$, so $Z\equiv h_1-xG$ with $G=\sum_{k\ge3}h_k$, $Z^{\langle-1\rangle}\equiv h_1+xG$, and $p_k\pl Z^{\langle-1\rangle}\equiv p_k$ for $k\ge2$. Therefore
\[
\cH^{\bS}_2\equiv\cL^{\bS}_2+x\,G\,\partial\cL^{\bS}_2\pmod{(u^{24},v^2)},
\]
and evaluating Theorem \ref{L2S} in this quotient up to degree $21$ gives the stated representation. Finally, $h_2[h_{10}]=\sum_{k\ \mathrm{even}}s_{(20-k,k)}$, so $h_2[e_{10}]=\omega(h_2[h_{10}])=\sum_{k\ \mathrm{even}}s_{2^k1^{20-2k}}$ since $e_{10}$ has even degree, and the Pieri rule gives the Schur expansion of $s_1\cdot h_2[e_{10}]$. Its dimension is $21\cdot\frac12\binom{20}{10}$.
\end{proof}

\begin{coro}\label{61}
$H^{24}(\Mbar_{2,21})$ is not spanned by Hodge classes, hence not by algebraic cycle classes, and in particular
$$RH^{24}(\Mbar_{2,21})\subsetneq H^{24}_{\mathrm{alg}}(\bcM_{2,21})\subsetneq H^{24}(\Mbar_{2,21})$$
where $RH^{24}(\bcM_{2,21})$ denotes the tautological subspace and $H^{24}_{\mathrm{alg}}(\bcM_{2,21})$ denotes the image of the cycle class map from the Chow group $A^{12}(\bcM_{2,21})$.
The same holds for $H^{24}(\bcM_{2,n})$ for every $n\ge 21$.
\end{coro}
\begin{proof}
Let $\overline{\mathcal B}_{2,1,20}\subset \bcM_{2,21}$ be the closure of the locus of bielliptic curves of genus $2$ with one marked point fixed by the bielliptic involution and $10$ pairs of marked points exchanged by it. Its dimension is $2+10=12$, so its class lies in $H^{24}(\bcM_{2,21})$, and it is not tautological by \cite[Theorem 1]{vZ}. Hence $RH^{24}(\bcM_{2,21})\ne H^{24}_{\mathrm{alg}}(\bcM_{2,21})$.
All algebraic classes lie in $H^{12,12}(\bcM_{2,21})$, but by Theorem~\ref{59}, the space $H^{24}(\bcM_{2,21})$ is not contained in $H^{12,12}(\bcM_{2,21})$.

For $n\ge 21$, the pullback by the forgetful morphism $\pi:\bcM_{2,n}\to\bcM_{2,21}$ is injective on cohomology because $\pi$ is surjective, and it preserves the Hodge types and the tautological subspaces. Hence $H^{24}(\bcM_{2,n})$ has classes of type $(23,1)$. Moreover $\pi^*[\overline{\mathcal B}_{2,1,20}]$ is not tautological: if $\pi':\bcM_{2,m+1}\to\bcM_{2,m}$ forgets the last marked point, then $\pi'_*(\psi_{m+1}\cdot\pi'^*\beta)=(2+m)\beta$ by the dilaton equation, so $\pi'^*\beta$ is not tautological if $\beta$ is not, because tautological rings are closed under products and $\pi'_*$; now factor $\pi$ into such morphisms.
\end{proof}


For $21\le n\le 30$, the off-diagonal even degree Hodge numbers are determined by our computation. Classes of types $(k,22+k)$ come from $\Sym^2\mathsf S_{12}(-k)$, and from $n=25$ on there are also classes from $\mathsf S_{12}\otimes\mathsf S_{16}$, of types $(1,27)$ and $(12,16)$ and their twists, and from $n=27$ on also from $\mathsf S_{12}\otimes\mathsf S_{18}$, for instance of type $(1,29)$. The following table lists $h^{p,q}(\bcM_{2,n})$ for $p<q$, $p+q$ even and $21\le n\le 26$; by Hodge symmetry $h^{q,p}=h^{p,q}$.

{\small
\begin{longtable}{@{}l p{0.88\textwidth}@{}}
\toprule
$n$ & off-diagonal even degree Hodge numbers with $p<q$\\
\midrule
$21$ & $h^{1,23}=1939938$\\
$22$ & $h^{1,23}=64017954$, $h^{2,24}=64017954$\\
$23$ & $h^{1,23}=1145210066$, $h^{2,24}=6216854644$, $h^{3,25}=1145210066$\\
$24$ & $h^{1,23}=14724129420$, $h^{2,24}=215953898160$, $h^{3,25}=215953898160$, $h^{4,26}=14724129420$\\
$25$ & $h^{1,23}=152149337340$, $h^{1,27}=49031400$, $h^{2,24}=4672457069280$, $h^{3,25}=14651830169130$, $h^{4,26}=4672457069280$, $h^{5,27}=152149337340$, $h^{12,16}=49031400$\\
$26$ & $h^{1,23}=1339895777220$, $h^{1,27}=1912224600$, $h^{2,24}=75672209132520$, $h^{2,28}=1912224600$, $h^{3,25}=535435604293590$, $h^{4,26}=535435604293590$, $h^{5,27}=75672209132520$, $h^{6,28}=1339895777220$, $h^{12,16}=1912224600$, $h^{13,17}=1912224600$\\
\\ \bottomrule
\end{longtable}}

\begin{prop}\label{n25}
As $\bS_{25}$-representations, $H^{1,27}(\bcM_{2,25})\cong H^{12,16}(\bcM_{2,25})$, with Frobenius characteristic $s_1e_{10}e_{14}$, of dimension $25\binom{24}{10}=49031400$.
\end{prop}
\begin{proof}
This is part of the computation of \S\ref{S7.1}, comparing Schur coefficients.
\end{proof}

As for $n=21$, the shape $s_1e_{10}e_{14}$ suggests products of the classes $\omega_{12}\in H^{11,0}(\bcM_{1,11})$ and $\omega_{16}\in H^{15,0}(\bcM_{1,15})$, which spans the sign representation \cite{Get98}, on the two elliptic components of curves in $\Delta_1$.

\begin{rema}
By Theorem \ref{L2S} and Corollary \ref{62}, we can determine all the $\bS_n$-representations $H^{p,p}(\bcM_{2,(0^+)^n})$ and $H^{p,p}(\bcM_{2,n})$ explicitly. Hence verifying the Hodge conjecture for $\bcM_{2,(0^+)^n}$, or equivalently for $\bcM_{2,n}$ by Proposition \ref{10}, now becomes an explicit finite count of cycles, which may be organized by isotypic components.
\end{rema}


\bigskip
\begin{thebibliography}{99}

\bibitem{AC} E.~Arbarello, M.~Cornalba, \emph{The Picard groups of the moduli spaces of curves}, Topology \textbf{26} (1987), 153--171.

\bibitem{Ber} J.~Bergstr\"om, \emph{$\bS_n$-equivariant Euler characteristics of moduli spaces of stable $n$-pointed curves of genus two}, table \texttt{eM2nbar.txt}.

\bibitem{Ber2} J.~Bergstr\"om,
\emph{Cohomology of moduli spaces of curves of genus two via point counts},
J. reine angew. Math. \textbf{622} (2008), 155--187.

\bibitem{CLP} S.~Canning, H.~Larson, S.~Payne, \emph{The eleventh cohomology group of $\Mbar_{g,n}$}, Forum Math. Sigma \textbf{11} (2023), e62.

\bibitem{Fulton} W.~Fulton, \emph{Intersection theory}, second edition, Springer, 1998.

\bibitem{FuHa} W.~Fulton, J.~Harris, \emph{Representation theory, a first course}, Graduate Texts in Mathematics 129, Springer, 2004.

\bibitem{Get95a} E.~Getzler, \emph{Operads and moduli spaces of genus $0$ Riemann surfaces}, in: The moduli space of curves (Texel Island, 1994), Progr. Math. \textbf{129}, Birkh\"auser, 1995, 199--230.

\bibitem{Get95b} E.~Getzler, \emph{Mixed Hodge structures of configuration spaces}, arXiv:alg-geom/9510018.

\bibitem{Get98} E.~Getzler, \emph{The semi-classical approximation for modular operads}, Comm. Math. Phys. \textbf{194} (1998), 481--492.

\bibitem{Get99} E.~Getzler, \emph{Euler characteristics of local systems on $\cM_2$}, Compositio Math. \textbf{132} (2002), 121--135; arXiv:math/9909171.

\bibitem{GK} E.~Getzler, M.~Kapranov, \emph{Modular operads}, Compositio Math. \textbf{110} (1998), 65--126.

\bibitem{GeP} E.~Getzler, R.~Pandharipande, \emph{The Betti numbers of $\Mbar_{0,n}(r,d)$}, J. Algebraic Geom. \textbf{15} (2006), 709--732.

\bibitem{GrP} T.~Graber, R.~Pandharipande, \emph{Constructions of nontautological classes on moduli spaces of curves}, Michigan Math. J. \textbf{51} (2003), 93--109.

\bibitem{GrHa} P.~Griffiths, J.~Harris, \emph{Principles of algebraic geometry}, Wiley and Sons, 1978.

\bibitem{Hassett} B.~Hassett, \emph{Moduli spaces of weighted pointed stable curves}, Adv. Math. \textbf{173} (2003), 316--352.

\bibitem{KSY} S.~Kannan, S.~Serpente, C.~Yun, \emph{Equivariant Hodge polynomials of heavy/light moduli spaces}, Forum Math. Sigma \textbf{12} (2024), e34.

\bibitem{KS} S.~Kannan, T.~Song, \emph{Graph enumeration for moduli spaces of curves and
maps}, arXiv:2509.18298.

\bibitem{Keel} S.~Keel, \emph{Intersection theory of moduli space of stable $n$-pointed curves of genus zero}, Trans. Amer. Math. Soc. \textbf{330} (1992), 545--574.

\bibitem{Ki} Y.-H.~Kiem, \emph{Hodge-Deligne and Poincar\'e polynomials of $\Mbar_{1,n}$}, arXiv:2609.12507.

\bibitem{KT} K.~Koike, I.~Terada, \emph{Young-diagrammatic methods for the representation theory of the classical groups of type $B_n$, $C_n$, $D_n$}, J. Algebra \textbf{107} (1987), 466--511.

\bibitem{LM} A.~Losev, Y.~Manin, \emph{New moduli spaces of pointed curves and pencils of flat connections}, Michigan Math. J. \textbf{48} (2000), 443--472.

\bibitem{Mac} I.~G.~Macdonald, \emph{Symmetric functions and Hall polynomials}, second edition, Oxford University Press, 1995.

\bibitem{Manin} Y.~Manin, \emph{Frobenius manifolds, quantum cohomology, and moduli spaces}, AMS Colloquium Publications 47, Providence, RI, 1999.

\bibitem{Pet0} D.~Petersen, \emph{The structure of the tautological ring in genus one}, Duke Math. J. \textbf{163} (2014), 777--793.

\bibitem{Pet} D.~Petersen, \emph{Cohomology of local systems on the moduli of principally polarized abelian surfaces}, Pacific J. Math. \textbf{275} (2015), 39--61.

\bibitem{Pet2} D.~Petersen, \emph{Tautological rings of spaces of pointed genus two curves of compact type}, Compos. Math. \textbf{152} (2016), 1398--1420.

\bibitem{PT} D.~Petersen, O.~Tommasi, \emph{The Gorenstein conjecture fails for the tautological ring of $\Mbar_{2,n}$}, Invent. Math. \textbf{194} (2013), 509--524.

\bibitem{vZ} J.~van Zelm, \emph{Nontautological bielliptic cycles}, Pacific J. Math. \textbf{294} (2018), 495--504.

\end{thebibliography}
\end{document}